\documentclass[11pt]{extarticle}
\usepackage{amsmath,amsthm,amssymb,amsbsy,amsopn,mathtools,comment,bbm}
\usepackage{graphicx,color}
\usepackage[bookmarks=true,bookmarksnumbered=true,bookmarkstype=toc]{hyperref}
\usepackage{mathptmx,bm}
\usepackage{algorithm}
\usepackage{algorithmicx,tabularx,tabu,longtable}
\usepackage{algpseudocode}
\usepackage{epstopdf}
\usepackage[titletoc]{appendix}
\usepackage{lineno} 
\usepackage{enumitem}
\setlist[itemize]{align=parleft,left=2pt..1em}

\def\argmin{\mathop{\rm argmin}}

\newtheorem{Def}{Definition}[section]
\newtheorem{theorem}[Def]{Theorem}

\newtheorem{remark}[Def]{Remark}

\newtheorem{assumption}[Def]{Assumption}

\usepackage{caption}
\usepackage{subcaption}

\theoremstyle{definition}

\makeatletter
\@addtoreset{equation}{section}
\makeatother

\title{A forward scheme with machine learning for forward-backward SDEs with jumps by decoupling jumps
\footnotetext[0]{
This work was partially supported by JSPS Grants-in-Aid for Scientific Research (Grant Numbers 20K22301 and 21K03347) and by JST PRESTO (Grant Number JPMJPR2029).}}
\author{\sc Reiichiro Kawai\footnote{Corresponding author. Email address: raykawai@g.ecc.u-tokyo.ac.jp. 
Affiliations: Graduate School of Arts and Sciences / Mathematics and Informatics Center, 
The University of Tokyo.},\, Riu Naito\footnote{Email address: riunaito@eco.u-toyama.ac.jp. Affiliations: Toyama University.}\, and Toshihiro Yamada\footnote{Email address: toshihiro.yamada@r.hit-u.ac.jp. Affiliations: Hitotsubashi University.}}
\date{}

\begin{document}
\maketitle

\begin{abstract}
\noindent 
Forward-backward stochastic differential equations (FBSDEs) have been generalized by introducing jumps for better capturing random phenomena, while the resulting FBSDEs are far more intricate than the standard one from every perspective.
In this work, we establish a forward scheme for potentially high-dimensional FBSDEs with jumps, taking a similar approach to [Bender and Denk, 117 (2007), {\it Stoch. Process. Their Appl.}, pp.1793-1812], with the aid of machine learning techniques for implementation.
The developed forward scheme is built upon a recursive representation that decouples random jumps at every step and converges exponentially fast to the original FBSDE with jumps, often requiring only a few iterations to achieve sufficient accuracy, along with the error bound vanishing for lower jump intensities.
The established framework also holds novelty in its neural network-based implementation of a wide class of forward schemes for FBSDEs, notably whether with or without jumps.
We provide an extensive collection of numerical results, showcasing the effectiveness of the proposed recursion and its corresponding forward scheme in approximating high-dimensional FBSDEs with jumps (up to 100-dimension) without directly handling the random jumps.

\vspace{0.2em}
\noindent {\it Keywords:} high-dimensional FBSDEs with jumps; partial integro-differential equations; forward scheme; machine learning; neural networks.

\noindent {\it 2020 Mathematics Subject Classification:} 60H30, 60G55, 35R09, 65C30, 68T07.
\end{abstract}

\section{Introduction}

To address problems in stochastic control, finance, and partial differential equations with more realistic features, forward-backward stochastic differential equations (FBSDEs) often require additional structures.
These extensions encompass jumps, switching, and reflection to better encapsulate realistic phenomena.
Nevertheless, FBSDEs endowed with such intricate features pose complexities beyond those of standard FBSDEs.
Evidently, all theoretical proofs and numerical approximation need to address irregularities within the underlying trajectories and law through more intricate procedures and restrictive conditions.

Among those additional features, the random jumps have often been incorporated into FBSDEs to literally allow the underlying trajectories to jump \cite{doi:10.1080/17442508.2018.1521808, doi:10.1080/17442508.2015.1090990}, for instance, as asset price dynamics do on a regular basis.
The associated semilinear parabolic partial integro-differential equation (PIDE) carries an additional integral term on top of the standard differential ones, reflecting the presence of additional random jumps.
Given the inherent challenges in handling standard FBSDEs, numerical approximation of FBSDEs with jumps naturally escalates in difficulty \cite{10.1214/23-PS18}.

To address difficulties arising from additional random jumps, a limited array of numerical methods has emerged in the literature.
These methods include discretization schemes tailored for treating jumps in FBSDEs, on the basis of Malliavin calculus \cite{Mall_Calc1}, by an explicit prediction-correction scheme \cite{EAJAM-6-253}, and via the shot noise representation of the Poisson random measure \cite{Massing2023}.
Other approaches incorporate forward schemes, built on Wiener chaos expansion and Picard iteration \cite{GEISS20162123}. 
A FBSDE driven by L\'evy processes is treated numerically with its small jumps approximated by the Brownian motion \cite{Fancy_BSDEs} and with 
Brownian and Poissonian components independently approximated by random walks \cite{lejay:inria-00357992}.
A Fourier-based method
is applied to utility indifference pricing problems formulated by FBSDEs with jumps \cite{doi:10.1137/130913183}.
Non-Poissonian local jump measures have also been investigated, such as a Fourier-based method \cite{doi:10.1137/16M1099005} and asymptotic expansions \cite{doi:10.1080/17442508.2018.1521808}.
Doubly reflected FBSDEs with jumps are numerically addressed with \cite{DUMITRESCU2016206} and without \cite{DUMITRESCU2016827} penalization.

In addition, in recent years, machine learning techniques have frequently demonstrated their effectiveness in the context of FBSDEs with jumps, particularly in high-dimensional settings, becoming increasingly predominant in the literature.
For instance, the deep BSDE solver \cite{Han8505}, further investigated in \cite{doi:10.1137/22M1478057, Han_2020}, has been tailored to tackle FBSDEs with jumps \cite{https://doi.org/10.48550/arxiv.2211.04349}.
Neural networks have often been applied in relevant contexts, such as approximating PIDEs with applications in insurance mathematics \cite{10.1007/978-3-030-99638-3_44} and solving high-dimensional optimal stopping problems with jumps in the
 energy market \cite{bayraktar2023neural}.
Most recently, a temporal difference learning method has been proposed for PIDEs, also built upon neural networks \cite{lu2023temporal}.

The aim of the present work is to establish a forward scheme for potentially high-dimensional FBSDEs with jumps, inspired by the forward scheme \cite{Picard_it1} proposed for standard FBSDEs (without jumps), following a similar approach to the existing decoupling principles for jump-diffusion processes \cite{doi:10.1080/17442508.2023.2259534} and regime switching \cite{QIU2022109301, regimebounds}.
The significance of the proposed forward scheme lies in the fact that each and every step of its recursion can be represented in terms of a standard parabolic linear partial differential equation (PDE, not PIDE), corresponding to diffusion processes without jumps.
This is achieved through a three-step process; first, by transforming the associated PIDE to a nonlinear PDE by decoupling random jumps (Theorem \ref{theorem Ytilde}); and second, by linearizing the resulting nonlinear PDEs through recursion (Theorem \ref{theorem section 3.2}); and finally, by constructing a forward scheme through discretization of the forward process that appears in the probabilistic representation of the recursion (Section \ref{subsection discretized recursion}).

The absence of random jumps or, equivalently, the corresponding integral terms of PIDEs in the final numerical procedure sets the proposed forward scheme far apart from the aforementioned existing methods for FBSDEs with jumps.
As a significant consequence, a variety of existing numerical methods can be employed in the established forward scheme as appropriate without directly handling the random jumps, such as approximation based on (semi-)analytic solutions, the iid Monte Carlo method, 
and machine learning methods, such as
the least squares Monte Carlo method (Section \ref{subsection LSMC}), and a neural network-based scheme (Section \ref{subsection neural network}).
In addition, the developed forward scheme approximates the original FBSDE with jumps exponentially fast through recursion along with the error bound that vanishes for lower jump intensities (Theorem \ref{thm1}).
The forward scheme here thus holds great potential for easing the numerical approximation of the realistic yet intricate FBSDEs with jumps in high-dimensional settings. 
Last but not least, the established framework also introduces intriguing novelty through its neural network-based implementation (Section \ref{subsection neural network}), accommodating a diverse range of forward schemes for FBSDEs, not only with jumps but remarkably also without jumps.

The rest of this paper is set out as follows.
In Section \ref{section problem formulation}, after defining the general notation, we formulate FBSDEs with jumps along with technical conditions and present their associated PIDEs.
In Section \ref{section recursive representation}, we decouple jumps from FBSDEs with jumps (Section \ref{subsection decoupling jumps}), and then develop a recursion by linear PDEs (Section \ref{subsection recursive representation}), supported by a theoretical analysis of convergence and error (Section \ref{subsection convergence and error analysis}).
In Section \ref{section forward scheme}, we construct a forward scheme so that the recursion established in Section \ref{section recursive representation} can be put into implementation by machine learning methods, along with illustrations (Figures \ref{Fig_Algo_Fwd_full}, \ref{Fig_Algo_Fwd_loc} and \ref{Fig_Algo_Bwd_loc}) provided to sketch its flow and contrast it with a typical backward scheme.
We present in Section \ref{section numerical illustrations} an extensive collection of numerical results to demonstrate the effectiveness of the proposed recursion and its corresponding forward scheme.
To maintain the flow of the paper, we collect proofs in Section \ref{section proofs} and the pseudocodes in the supplementary materials.

\section{Problem formulation}
\label{section problem formulation}

We first prepare the notation that will be used throughout the paper and introduce the underlying stochastic process.
We denote the set of real numbers by $\mathbb{R}$, the set of natural numbers (without zero) by $\mathbb{N}:= \{1,2,\cdots\}$, and then $\mathbb{N}_0:=\mathbb{N}\cup \{0\}$ (including zero).
We denote by $\mathbb{R}^d$ and $|\cdot|$, respectively, the $d$-dimensional Euclidean space and norm, with $\mathbb{R}^d_0:=\mathbb{R}^d\setminus \{0\}$.
We write $\mathbb{I}_d$ for the identity matrix in $\mathbb{R}^{d\times d}$, $\mathbbm{1}_d$ for the vector in $\mathbb{R}^d$ with all unit-valued components, and $0_d$ for the zero vector in $\mathbb{R}^d$.
The scalar one and zero are simply denoted by $1$ and $0$, respectively.
As usual, we write $(\cdot)^+:=\max\{\cdot,0\}$ and ${\rm diag}({\bf x}):={\rm diag}(x_1,\cdots,x_d)$ for ${\bf x}=(x_1,\cdots,x_d)^{\top}$.
Throughout, we reserve $T$ for the terminal time of the interval of interest, that is, $[0,T]$.
Let $(\Omega,{\cal F},({\cal F}_s)_{s\in [0,T]}, \mathbb{P})$ be a filtered probability space, where $\textstyle{{\cal F}_{s+}:=\bigcap_{\epsilon>0}{\cal F}_{s+\epsilon}={\cal F}_s}$ for all $s>0$ and ${\cal F}_0$ contains null set.

The stochastic system of our interest is the forward-backward stochastic differential equation (FBSDE) with jumps, formulated as follows: for $(t,{\bf x})\in (0,T]\times \mathbb{R}^{d_{\bf x}}$ and $s\in [t,T]$,
\begin{equation}\label{FBSDE_with_jump}
\begin{dcases}
dX_s^{t,{\bf x}}=b(s,X_s^{t,{\bf x}})ds+\sigma(s,X_s^{t,{\bf x}})dB_s+\int_{\mathbb{R}_0^{d_{\bf z}}}h(s,X_{s-}^{t,{\bf x}},{\bf z})\widetilde{\mu}(d{\bf z},ds;X_{s-}^{t,{\bf x}}),\\
dY_s^{t,{\bf x}}=-f(s,X_s^{t,{\bf x}},Y_s^{t,{\bf x}},Z_s^{t,{\bf x}},\Gamma_s^{t,{\bf x}})ds+Z_s^{t,{\bf x}}dB_s+\int_{\mathbb{R}_0^{d_{\bf z}}}U_{s-}^{t,{\bf x}}({\bf z})\widetilde{\mu}(d{\bf z},ds;X_{s-}^{t,{\bf x}}),\\ 
X_{t}^{t,{\bf x}}={\bf x},\quad Y_T^{t,{\bf x}}=\Phi(X_T^{t,{\bf x}}),\quad  \Gamma_s^{t,{\bf x}}=\int_{\mathbb{R}_0^{d_{\bf z}}}U_s^{t,{\bf x}}({\bf z})\lambda(s,X_s^{t,{\bf x}})\nu(d{\bf z}),
\end{dcases}
\end{equation}
with the goal of finding the initial state of the backward process $Y_t^{t,{\bf x}}$.

In the remainder of this section, we define every component of the FBSDE \eqref{FBSDE_with_jump} and impose relevant technical conditions.
In turn, $b:[0,T]\times \mathbb{R}^{d_{\bf x}}\to \mathbb{R}^{d_{\bf x}}$, $\sigma:[0,T]\times \mathbb{R}^{d_{\bf x}}\to \mathbb{R}^{d_{\bf x}\times d_0}$, $\{B_s: s\in [0,T]\}$ is the $(\mathcal{F}_s)_{s\in [0,T]}$-adapted standard Brownian motion in $\mathbb{R}^{d_0}$, $h:[0,T]\times\mathbb{R}^{d_{\bf x}}\times\mathbb{R}_0^{d_{\bf z}}\to\mathbb{R}^{d_{\bf x}}$, $f:[0,T]\times \mathbb{R}^{d_{\bf x}}\times \mathbb{R}^{d_{\bf y}}\times \mathbb{R}^{d_{\bf y}\times d_0}\times \mathbb{R}^{d_{\bf y}}\to \mathbb{R}^{d_{\bf y}}$, 
and $\Phi:\mathbb{R}^{d_{\bf x}}\to \mathbb{R}^{d_{\bf y}}$. 
To define $\widetilde{\mu}$ in \eqref{FBSDE_with_jump}, we first prepare $\mu$, which is a function mapping $\Omega\times [0,T]\times\mathbb{R}^{d_{\bf z}}_0\times\mathbb{R}^{d_{\bf x}}$ to $\mathbb{R}$ such that
\begin{itemize}
\setlength{\parskip}{0cm}
\setlength{\itemsep}{0cm}
\item for every $\omega\in\Omega$ and ${\bf x}\in\mathbb{R}^{d_{\bf x}}$, $\mu(\omega,\cdot;{\bf x})$
is a $\sigma$-finite measure on $[0,T]\times\mathbb{R}_0^{d_{\bf z}}$, 
\item for every $A\in {\cal B}([0,T])\otimes{\cal B}(\mathbb{R}_0^{d_{\bf z}})$ and ${\bf x}\in\mathbb{R}^{d_{\bf x}}$, $\mu(\cdot,A;{\bf x}):\Omega\rightarrow\mathbb{N}_0$ is a random variable on $(\Omega,{\cal F},\mathbb{P})$,
\end{itemize}
with abbreviation $\mu(ds,d{\bf z};{\bf x})(\omega):=\mu(\omega,ds,d{\bf z};{\bf x})$.
With ${\bf x}\in\mathbb{R}^{d_{\bf x}}$ fixed, $\mu(\cdot,\cdot;{\bf x})(\cdot)$ is a Poisson random measure on $\Omega\times [0,T]\times\mathbb{R}_0^{d_{\bf z}}$ with compensator $\lambda(s,{\bf x})\nu(d{\bf z})ds$, where $\nu$ is a L\'evy measure $\nu$ on $\mathbb{R}_0^{d_{\bf z}}$ with normalization $\nu(\mathbb{R}_0^{d_{\bf z}})=1$ and $\lambda$ is a function $[0,T]\times\mathbb{R}^{d_{\bf x}}\rightarrow [0,+\infty)$.
Clearly, the normalization $\nu(\mathbb{R}_0^{d_{\bf z}})=1$ loses no generality in the presence of the intensity function $\lambda(s,{\bf x})$.
With those components, we denote by $\widetilde{\mu}(d{\bf z},ds;{\bf x})$ a compensated Poisson random measure on $\mathbb{R}_0^{d_{\bf z}}\times [0,T]$, independent of the standard Brownian motion $\{B_s: s\in [0,T]\}$ if the argument ${\bf x}$ is fixed in $\mathbb{R}^{d_{\bf x}}$, defined as
\begin{equation}\label{compensator jumps}
\widetilde{\mu}(d{\bf z},ds;{\bf x})(\omega):=\widetilde{\mu}(\omega,d{\bf z},ds;{\bf x}):=\mu(d{\bf z},ds;{\bf x})(\omega)-\lambda(s,{\bf x})\nu(d{\bf z})ds,
\end{equation}
With the argument ${\bf x}$ governed by the marginal of the forward process as in the formulation \eqref{FBSDE_with_jump}, the random measure $\mu(d{\bf z},ds;X_{s-}^{t,{\bf x}})$ in \eqref{compensator jumps} is neither Poissonian \cite{doi:10.1137/16M1099005} nor independent of the standard Brownian motion $\{B_s:\,s\in [0,T]\}$, unless the intensity function $\lambda(s,{\bf x})$ is independent of the argument ${\bf x}$.
Before closing this section, we summarize technical conditions on the FBSDE with jumps \eqref{FBSDE_with_jump}.

\begin{assumption}{\rm \label{standing assumption}
Throughout the paper, we assume that the following conditions hold true:
\begin{itemize}
\setlength{\parskip}{0cm}
\setlength{\itemsep}{0cm}
    \item $b$ 
    and $\sigma$
    are continuously differentiable in ${\bf x}$ with bounded zeroth- and first-order partial derivatives in ${\bf x}$.
    \item $f$ 
    and $\Phi$
    are uniformly Lipschitz continuous in $({\bf x},{\bf y},{\bf z},u)$, and 
    continuously differentiable in $({\bf x},{\bf y},{\bf z})$ with bounded zeroth- and first-order partial derivatives in $({\bf x},{\bf y},{\bf z})$.
    \item $b,\sigma,f$ are uniformly $(1/2)$-H\"older continuous in $t$.
    \item $\lambda$ 
    is non-negative, continuous in all their arguments, with $\sup_{(s,{\bf x})\in [0,T]\times \mathbb{R}^{d_{\bf x}}} \lambda(s,{\bf x})=\eta$ for some non-negative $\eta$.
    \item $v$ is a L\'evy measure on $\mathbb{R}_0^{d_{\bf z}}$ satisfying  $v(\mathbb{R}_0^{d_{\bf z}})=1$.
    \item $h$
    satisfies $|h(t,{\bf x},{\bf z})|\leq C\min\{1,|{\bf z}|\}$ and $|h(t,{\bf x},{\bf z})-h(t,{\bf x}',{\bf z})|\leq C|{\bf x}-{\bf x}'|\min\{1,|{\bf z}|\}$ for all $(t,{\bf x},{\bf x}',{\bf z})\in[0,T]\times\mathbb{R}^{d_{\bf x}}\times\mathbb{R}^{d_{\bf x}}\times\mathbb{R}_0^{d_{\bf z}}$ for some $C>0$.
    \item $\sigma$ is uniformly elliptic on $[0,T]\times\mathbb{R}^{d_{\bf x}}$.\qed
  \end{itemize}
}\end{assumption}

We note that the primary focus of the forward scheme developed in what follows lies in the presence of random jumps (that is, $\eta>0$), whereas the scheme can also handle the case of no jumps (that is, $\eta=0$).
We illustrate this point in Remark \ref{remark bender denk 2}.

The aim of the present work is to find the initial state of the backward process $Y_t^{t,{\bf x}}$, which can be expressed as:
\begin{equation}\label{udef}
Y_t^{t,{\bf x}}=\mathbb{E}\left[Y_t^{t,{\bf x}}\right]= \mathbb{E}\left[\Phi(X_T^{t,{\bf x}})+\int_t^T f(s,X_s^{t,{\bf x}},Y_s^{t,{\bf x}},Z_s^{t,{\bf x}},\Gamma_s^{t,{\bf x}}) ds\right].
\end{equation}
To this end, consider the partial integro-differential equation (PIDE): 
\begin{equation}\label{PIDE}
\begin{dcases}
\frac{\partial}{\partial s} u(s,{\bf x})+\left\langle b(s,{\bf x}),\nabla  u(s,{\bf x})\right\rangle +\frac{1}{2}{\rm tr}\left[(\sigma(s,{\bf x}))^{\otimes 2}{\rm Hess}(u(s,{\bf x}))\right]+f(s,{\bf x},u(s,{\bf x}),(\nabla u \sigma)(s,{\bf x}),\gamma(s,{\bf x};u))\\
\qquad +\int_{\mathbb{R}_0^{d_{\bf z}}}(u(s,{\bf x}+h(s,{\bf x},{\bf z}))-u(s,{\bf x})-\left\langle h(s,{\bf x},{\bf z}),\nabla u(s,{\bf x})\right\rangle)\lambda(s,{\bf x})\nu(d{\bf z})=0, \quad (s,{\bf x})\in [t,T)\times\mathbb{R}^{d_{\bf x}},\\
u(T,{\bf x})=\Phi({\bf x}),\quad {\bf x}\in\mathbb{R}^{d_{\bf x}},
\end{dcases}
\end{equation}
where the gradient and Hessian are taken with respect to the state variable ${\bf x}$, and
\begin{equation}\label{def of gamma}
\gamma(s,{\bf x};\phi):=\int_{\mathbb{R}_0^{d_{\bf z}}}\left(\phi(s,{\bf x}+h(s,{\bf x},{\bf z}))-\phi(s,{\bf x})\right)\lambda(s,{\bf x})\nu(d{\bf z}),
\end{equation}
for $\phi: [0,T]\times \mathbb{R}^{d_{\bf x}}\to \mathbb{R}$.
In the PIDE \eqref{PIDE}, we have written $(\nabla v \sigma)(s,{\bf x}):=(\nabla v_1(s,{\bf x}),\nabla v_2(s,{\bf x}),\cdots,\nabla v_{d_{\bf y}}(s,{\bf x}))^{\top}\sigma(s,{\bf x})\in \mathbb{R}^{d_{\bf y}\times d_0}$, where $v_k$ denotes the $k$ component of the $\mathbb{R}^{d_{\bf y}}$-valued function $v$, that is, $v(s,{\bf x}):=(v_1(s,{\bf x}),v_2(s,{\bf x}),\cdots,v_{d_{\bf y}}(s,{\bf x}))^{\top}$.
It is then well known that the solution of the FBSDE \eqref{FBSDE_with_jump} can be represented in terms of the solution to the PIDE \eqref{PIDE} by the nonlinear Feynman Kac formula, as follows:
\begin{equation}\label{nFK_jump0}
Y_s^{t,{\bf x}}=u(s,X_s^{t,{\bf x}}),\ \ Z_s^{t,{\bf x}}=
(\nabla u \sigma)(s,X_s^{t,{\bf x}}),\ \ U_s^{t,{\bf x}}({\bf z})=u(s,X_s^{t,{\bf x}}+h(s,X_s^{t,{\bf x}},{\bf z}))-u(s,X_s^{t,{\bf x}}),
\end{equation}
for all $s\in [t,T]$.
In particular, the initial state of the backward process $Y_t^{t,{\bf x}}$ can thus be represented as  
\begin{equation}\label{nFK_jump}
Y_t^{t,{\bf x}}=u(t,{\bf x}),
\end{equation}
in addition to the representation \eqref{udef}.

\begin{remark}{\rm \label{remark bender denk 1}
Before delving into the main sections, let us motivate the present work based upon what we have prepared so far.
As outlined in the introductory section, our aim is to establish a recursion and its corresponding forward scheme for the initial state $Y_t^{t,{\bf x}}$ of the FBSDE with jumps \eqref{FBSDE_with_jump} or, equivalently, the solution $u(t,{\bf x})$ to the PIDE \eqref{PIDE} 
through a recursive representation where random jumps are somehow dropped.
Given this context, one might wonder about the implications of directly employing the existing forward scheme \cite{Picard_it1} proposed for the standard FBSDE (without jumps) for achieving this goal.
In short, the forward scheme \cite{Picard_it1} produces the Picard iteration $\{v_m\}_{m\in \mathbb{N}_0}$ involving random jumps, as follows:
\begin{equation}\label{v0}
v_0(t,{\bf x})=\mathbb{E}\left[\Phi(X_T^{t,{\bf x}})\right],
\end{equation}
and then, recursively for $m\in \mathbb{N}$,
\begin{equation}\label{vm}
v_m(t,{\bf x})=\mathbb{E}\left[\Phi(X_T^{t,{\bf x}})+\int_t^Tf(s,X_s^{t,{\bf x}},v_{m-1}(s,X_s^{t,{\bf x}}),(\nabla v_{m-1} \sigma)(s,X_s^{t,{\bf x}}),\gamma(s,X_s^{t,{\bf x}};v_{m-1}))ds \right],
\end{equation}
that is, all formulated upon the forward process with jumps $\{X_s^{t,x}:\, s\in [t,T]\}$ in the FBSDE \eqref{FBSDE_with_jump}.
In view of this result, we aim to establish a recursion and its corresponding forward scheme for potentially high-dimensional FBSDEs with jumps \eqref{FBSDE_with_jump}, but here without requiring random jumps in the forward process.
\qed}\end{remark}

\section{The recursive representation}
\label{section recursive representation}

We are now ready to establish a recursive representation for the initial state $Y_t^{t,{\bf x}}$ of the FBSDE with jumps \eqref{FBSDE_with_jump} or, equivalently, the solution $u(t,{\bf x})$ to the PIDE \eqref{PIDE} through the identity \eqref{nFK_jump}, with an aim towards numerical approximation,  particularly in high-dimensional settings, with the help of machine learning techniques (Section \ref{section forward scheme}).
To this end, we proceed in three steps: first, we transform the PIDE \eqref{PIDE} to a nonlinear PDE by decoupling random jumps (Section \ref{subsection decoupling jumps}); next, we linearize the resulting nonlinear PDEs (Section \ref{subsection recursive representation}); and finally, we conduct a convergence and error analysis (Section \ref{subsection convergence and error analysis}).

\subsection{Decoupling jumps from FBSDEs with jumps}
\label{subsection decoupling jumps}

Consider the FBSDE in the absence of random jumps, defined as follows: for $s\in [t,T]$,
\begin{equation}\label{FBSDE_without_jump}
\begin{dcases}
d\widetilde{X}_s^{t,{\bf x}} = b(s,\widetilde{X}_s^{t,{\bf x}})ds+ \sigma (s,\widetilde{X}_s^{t,{\bf x}})dB_s-\int_{\mathbb{R}_0^{d_{\bf z}}}h(s,\widetilde{X}_s^{t,{\bf x}},{\bf z})\lambda(s,\widetilde{X}_s^{t,{\bf x}})\nu(d{\bf z})ds,\\
d\widetilde{Y}_s^{t,{\bf x}}=-f(s,\widetilde{X}_s^{t,{\bf x}},\widetilde{Y}_s^{t,{\bf x}},\widetilde{Z}_s^{t,{\bf x}},\widetilde{\Gamma}_s^{t,{\bf x}})ds 
+\widetilde{Z}_s^{t,{\bf x}}dB_s
-\widetilde{\Gamma}_s^{t,{\bf x}}ds,\\
\widetilde{X}_{t}^{t,{\bf x}}={\bf x},\quad \widetilde{Y}_T^{t,{\bf x}}=\Phi(\widetilde{X}_T^{t,{\bf x}}),\quad \widetilde{\Gamma}_s^{t,{\bf x}}=\int_{\mathbb{R}_0^{d_{\bf z}}}\widetilde{U}_s^{t,{\bf x}}({\bf z})\lambda(s,\widetilde{X}_s^{t,{\bf x}})\nu(d{\bf z}),
\end{dcases}
\end{equation}
to which we refer as the FBSDE without jumps.
For clarification, the term ``without jumps'' does not imply the complete absence of a jump component, but rather the absence of random jumps, specifically.
In light of the original FBSDE with jumps \eqref{FBSDE_with_jump} and the compensated random measure \eqref{compensator jumps}, all random jumps (arising from the random counting measure $\mu(d{\bf z},ds;{\bf x})$) are suppressed, whereas all terms associated with its compensator $\lambda(s,{\bf x})\nu(d{\bf z})ds$ are retained throughout the formulation \eqref{FBSDE_without_jump}.

The FBSDE without jumps \eqref{FBSDE_without_jump} is closely related to the original FBSDE \eqref{FBSDE_with_jump} and the solution $u$ to its associated PIDE \eqref{PIDE}, as follows.

\begin{theorem}\label{theorem Ytilde}
It holds that for every $(t,{\bf x})\in [0,T]\times \mathbb{R}^{d_{\bf x}}$,
\[
\widetilde{Y}_t^{t,{\bf x}}=u(t,{\bf x}).
\]
\end{theorem}

In light of the representation \eqref{nFK_jump}, we have thus derived the identity 
\begin{equation}\label{Ytilde = Y}
\widetilde{Y}_t^{t,{\bf x}}=u(t,{\bf x})=Y_t^{t,{\bf x}}.
\end{equation}
In short, we have matched the initial state $\widetilde{Y}_t^{t,{\bf x}}$ of the FBSDE without jumps \eqref{FBSDE_with_jump} to the initial state $Y_t^{t,{\bf x}}$ of the FBSDE with jumps \eqref{FBSDE_with_jump} by dropping all random jumps and adjusting the drift coefficient and driver appropriately, as follows:
\begin{eqnarray*}
 b(s,X_s^{t,{\bf x}}) &
 \Longrightarrow & b(s,\widetilde{X}_s^{t,{\bf x}})-\int_{\mathbb{R}_0^{d_{\bf z}}}h(s,\widetilde{X}_s^{t,{\bf x}},{\bf z})\lambda(s,\widetilde{X}_s^{t,{\bf x}})\nu(d{\bf z}),\\
 f(s,X_s^{t,{\bf x}},Y_s^{t,{\bf x}},Z_s^{t,{\bf x}},\Gamma_s^{t,{\bf x}})&
 \Longrightarrow &
 f(s,\widetilde{X}_s^{t,{\bf x}},\widetilde{Y}_s^{t,{\bf x}},\widetilde{Z}_s^{t,{\bf x}},\widetilde{\Gamma}_s^{t,{\bf x}})
+\widetilde{\Gamma}_s^{t,{\bf x}}.
\end{eqnarray*}

From a numerical perspective, thanks to the identity \eqref{Ytilde = Y}, it suffices to find the initial state $\widetilde{Y}_t^{t,{\bf x}}$ of the FBSDE without jumps \eqref{FBSDE_without_jump} when computing the solution $u(t,{\bf x})$, rather than $Y_t^{t,{\bf x}}$ of the original FBSDE with jumps \eqref{FBSDE_with_jump}.
It is worth noting that the identity \eqref{Ytilde = Y} can also be of independent theoretical interest, in the sense that the two FBSDEs \eqref{FBSDE_with_jump} and \eqref{FBSDE_without_jump} exhibit markedly different trajectories (specifically, with or without discontinuities) during intermediate times.

\subsection{Recursion by linear PDEs}
\label{subsection recursive representation}

Even after random jumps have been completely suppressed, dealing with the FBSDE \eqref{FBSDE_without_jump} numerically remains often challenging due to its high nonlinearity, as evidenced by the long-term and ongoing development of various numerical methods for FBSDEs \cite{10.1214/23-PS18}.
In particular, to avoid nested conditional expectations, which are typical in discretization schemes, we linearize the FBSDE without jumps \eqref{FBSDE_without_jump} through recursion, following a similar approach to the recursion developed for FBSDEs in \cite{Picard_it1} and for jump-diffusion processes \cite{doi:10.1080/17442508.2023.2259534} and regime switching \cite{QIU2022109301, regimebounds}.
To be precise, consider a sequence of linear PDEs, starting with
\begin{equation}\label{PIDE_0}
\begin{dcases}
\frac{\partial}{\partial s}w_0(s,{\bf x})
+\left\langle b(s,{\bf x})-\int_{\mathbb{R}_0^{d_{\bf z}}}h(s,{\bf x},{\bf z})\lambda(s,{\bf x})\nu(d{\bf z}), \nabla w_0(s,{\bf x})\right\rangle+\frac{1}{2}{\rm tr}\left[(\sigma(s,{\bf x}))^{\otimes 2}{\rm Hess}(w_0(s,{\bf x}))\right]\\
\qquad\qquad\qquad\qquad\qquad\qquad\qquad\qquad\qquad\qquad\qquad\qquad\qquad =0, \quad(s,{\bf x})\in [t,T)\times\mathbb{R}^{d_{\bf x}},\\
w_0(T,{\bf x})=\Phi({\bf x}),\quad {\bf x}\in\mathbb{R}^{d_{\bf x}},
\end{dcases}
\end{equation}
and then, recursively for $m\in\mathbb{N}$, 
\begin{equation}\label{PIDE_m}
\begin{dcases}
\frac{\partial}{\partial s}w_m(s,{\bf x})
+\left\langle b(s,{\bf x})-\int_{\mathbb{R}_0^{d_{\bf z}}}h(s,{\bf x},{\bf z})\lambda(s,{\bf x})\nu(d{\bf z}), \nabla w_m(s,{\bf x})\right\rangle+\frac{1}{2}{\rm tr}\left[(\sigma(s,{\bf x}))^{\otimes 2}{\rm Hess}(w_m(s,{\bf x}))\right]\\
\qquad +f(s,{\bf x},w_{m-1}(s,{\bf x}),(\nabla w_{m-1} \sigma)(s,{\bf x}),\gamma(s,{\bf x};w_{m-1}))+\gamma(s,{\bf x};w_{m-1})=0, \quad
(s,{\bf x})\in [t,T)\times\mathbb{R}^{d_{\bf x}},\\
w_m(T,{\bf x})=\Phi({\bf x}),\quad {\bf x}\in\mathbb{R}^{d_{\bf x}}.
\end{dcases}
\end{equation}

The PDEs \eqref{PIDE_0} and \eqref{PIDE_m} are not only linear and recursive, as is evident, but those are also profound for the reason that those admit the following probabilistic representations, specifically formulated on the forward process without jumps $\{\widetilde{X}_s^{t,{\bf x}}:\,s\in [t,T]\}$ appearing in the FBSDE \eqref{FBSDE_without_jump}.

\begin{theorem}\label{theorem section 3.2}
If $w_m\in \mathcal{C}^{1,2}([0,T]\times\mathbb{R}^{d_{\bf x}};\mathbb{R}^{d_{\bf y}})$, $\sup_{(t,{\bf x})\in [0,T]\times \mathbb{R}^{d_{\bf x}}}|w_m(t,{\bf x})|\leq C$, and $\sup_{(t,{\bf x})\in [0,T]\times \mathbb{R}^{d_{\bf x}}}|\nabla w_m(t,{\bf x})|\leq C$ for some universal constant $C>0$ and all $m\in\mathbb{N}_0$, then it holds that for every $(t,{\bf x})\in [0,T]\times \mathbb{R}^{d_{\bf x}}$,
\begin{equation}
w_0(t,{\bf x})=\mathbb{E}\left[\Phi(\widetilde{X}_T^{t,{\bf x}})\right],\label{w_0_E}
\end{equation}
and
\begin{equation}
w_m(t,{\bf x})=\mathbb{E}\left[\Phi(\widetilde{X}_T^{t,{\bf x}})+\int_t^T f(s,\widetilde{X}_s^{t,{\bf x}},w_{m-1}(s,\widetilde{X}_s^{t,{\bf x}}),(\nabla w_{m-1} \sigma)(s,\widetilde{X}_s^{t,{\bf x}}),\gamma(s,\widetilde{X}_s^{t,{\bf x}};w_{m-1}))ds 
+\int_t^T \gamma(s,\widetilde{X}_s^{t,{\bf x}};w_{m-1})ds\right],\label{w_m_E}
\end{equation}
for all $m\in\mathbb{N}$.
\end{theorem}

It is worth stressing again that the absence of random jumps throughout the probabilistic representations \eqref{w_0_E} and \eqref{w_m_E} contrasts sharply with the presence of random jumps in the aforementioned recursion upon \eqref{v0} and \eqref{vm}.
Evidently, this difference has the potential to significantly impact the computation.

\subsection{Convergence and error analysis}
\label{subsection convergence and error analysis}

In addition to the absence of random jumps in the recursion $\{w_m\}_{m\in\mathbb{N}_0}$, the following theorem ensures its convergence and provides an error estimate relative to the solution to the original PIDE \eqref{PIDE}.
In short, the convergence is uniform in both time and state variables, occurring exponentially fast through recursion for the FBSDE with jump \eqref{FBSDE_with_jump} on the basis of an lower error bound for decreased jump intensities.
To maintain the primary focus on the FBSDE with jumps \eqref{FBSDE_with_jump} (that is, $\eta>0$), we do not address the case of zero jump intensity (that is, $\eta=0$) here but provide some insights in Remark \ref{remark bender denk 2}. 

\begin{theorem}\label{thm1}
It holds 
that for every $m\in \mathbb{N}$ and $\eta>0$,
\[
  \sup_{(t,{\bf x})\in [0,T]\times \mathbb{R}^{d_{\bf x}}}\left|(u-w_m)(t,{\bf x})\right|\le C2^{-m}(1\land\eta^m),
\]
for some positive constant $C$ that depends only on the Lipschitz constant of $f$ and the bounds of $u$, $w_0$, $b$, $\sigma$
and $h$.
\end{theorem}

To summarize the theoretical developments so far, the established recursion for FBSDE with jumps is free from random jumps through linearized recursion, while still converging exponentially fast to the solution of original interest, supported by an error bound that vanishes with respect to lower jump intensities.
We stress that this linearization may enable the use of various numerical methods, such as approximation based on (semi-)analytic solutions, the iid Monte Carlo method, the least squares Monte Carlo method (Section \ref{subsection LSMC}), and neural network-based schemes (Section \ref{subsection neural network}), as appropriate.

\begin{remark}{\rm \label{remark bender denk 2}
If the jump component is completely suppressed from the FBSDE with jumps \eqref{FBSDE_with_jump} by setting the jump intensity $\lambda(\cdot,\cdot)\equiv 0$ (so that $\eta=0$), effectively dropping all terms involving random jumps and the compensator, then the recursion $\{w_m\}$ developed in this work remains valid and closely resembles the recursion proposed for FBSDEs without jumps \cite{Picard_it1}, but with a slight difference, which we clarify below. 
With the jump intensity suppressed, the recursion $\{w_m\}_{m\in\mathbb{N}_0}$ developed in \eqref{PIDE_0} and \eqref{PIDE_m} reduces to 
\[
\frac{\partial}{\partial s}w_0(s,{\bf x})
+\langle b(s,{\bf x}), \nabla w_0(s,{\bf x})\rangle+\frac{1}{2}{\rm tr}\left[(\sigma(s,{\bf x}))^{\otimes 2}{\rm Hess}(w_0(s,{\bf x}))\right]=0,
\]
and then, recursively for $m\in \mathbb{N}$,
\[
\frac{\partial}{\partial s}w_m(s,{\bf x})
+\left\langle b(s,{\bf x}), \nabla w_m(s,{\bf x})\right\rangle+\frac{1}{2}{\rm tr}\left[(\sigma(s,{\bf x}))^{\otimes 2}{\rm Hess}(w_m(s,{\bf x}))\right]
+f(s,{\bf x},w_{m-1}(s,{\bf x}),(\nabla w_{m-1} \sigma)(s,{\bf x}),0)=0,
\]
with $w_m(T,{\bf x})=\Phi({\bf x})$ for all $m\in\mathbb{N}_0$, along with their probabilistic representations \eqref{w_0_E} and \eqref{w_m_E} resulting in 
\[
w_0(t,{\bf x})=\mathbb{E}\left[\Phi(\widetilde{X}_T^{t,{\bf x}})\right],\quad w_m(t,{\bf x})=\mathbb{E}\left[\Phi(\widetilde{X}_T^{t,{\bf x}})+\int_t^T f(s,\widetilde{X}_s^{t,{\bf x}},w_{m-1}(s,\widetilde{X}_s^{t,{\bf x}}),(\nabla w_{m-1} \sigma)(s,\widetilde{X}_s^{t,{\bf x}}),0)ds \right],
\]
for $m\in\mathbb{N}$.
For comparison with the existing recursion developed in \cite{Picard_it1}, observe that the reduced expression for the recursion above can be represented in the form of a sequence of FBSDEs as $\widetilde{Y}_t^{t,{\bf x},m}=w_m(t,{\bf x})$ and $\widetilde{Z}_t^{t,{\bf x},m}=(\nabla w_m\sigma)(t,{\bf x})$, on the basis of the recursive definition for the (originally backward) component, as follows: 
\begin{equation}\label{FBSDE_without_jump_eta_0}
d\widetilde{Y}_s^{t,{\bf x},0}=\widetilde{Z}_s^{t,{\bf x},0}dB_s,\quad 
d\widetilde{Y}_s^{t,{\bf x},m}=-f(s,\widetilde{X}_s^{t,{\bf x}},\widetilde{Y}_s^{t,{\bf x},m-1},\widetilde{Z}_s^{t,{\bf x},m-1},0)ds 
+\widetilde{Z}_s^{t,{\bf x},m}dB_s,\quad s\in [t,T],
\end{equation}
with $\widetilde{Y}_T^{t,{\bf x},m}=\Phi(\widetilde{X}_T^{t,{\bf x}})$ for all $m\in\mathbb{N}_0$, upon the forward process formulated just as in \eqref{FBSDE_without_jump}, that is, $d\widetilde{X}_s^{t,{\bf x}} = b(s,\widetilde{X}_s^{t,{\bf x}})ds+ \sigma (s,\widetilde{X}_s^{t,{\bf x}})dB_s$ with $\widetilde{X}_t^{t,{\bf x}}={\bf x}$. 
That is to say, it has been found that the recursion \eqref{FBSDE_without_jump_eta_0} differs from the existing recursion developed in \cite{Picard_it1} only in its initial (zeroth) process, which can be represented as $d\widetilde{Y}_s^{t,{\bf x},0}=-f(s,\widetilde{X}_s^{t,{\bf x}},0,0,0)ds +\widetilde{Z}_s^{t,{\bf x},0}dB_s$ in  \cite{Picard_it1}.

Given that the only difference lies in the initial process with all subsequent iteration mechanisms in common, it is not surprising at all that there is no essential discrepancy in terms of convergence and its rate.
In order to avoid clouding the main scope of the proposed recursion, we have excluded the case of zero jump intensity (that is, $\eta=0$) from the statement of Theorem \ref{thm1}, whereas we derive its error estimate as $\sup_{(t,{\bf x})\in [0,T]\times \mathbb{R}^{d_{\bf x}}}\left|(u-w_m)(t,{\bf x})\right|\le C2^{-m}$ (simultaneously with the primary case of $\eta>0$) in the proof (Section \ref{section proofs}).
In essence, the proposed recursion can thus be considered a generalization of the one in \cite{Picard_it1} for handling additional random jumps.
\qed}\end{remark}

\section{The forward scheme with machine learning techniques}
\label{section forward scheme}

We have now established a recursive representation for the initial state $Y_t^{t,{\bf x}}$ of the FBSDE with jumps \eqref{FBSDE_with_jump} or, equivalently, the solution $u(t,{\bf x})$ to the PIDE \eqref{PIDE} in the form of a sequence of linear PDEs \eqref{PIDE_0} and \eqref{PIDE_m} or, equivalently, their probabilistic representations \eqref{w_0_E} and \eqref{w_m_E}.
The aim of the present section is to put the developed recursion into implementation in the form of a forward scheme with the aid of machine learning techniques.
In what follows, without loss of generality, we restrict ourselves to approximating the solution $u(0,{\bf x})$ at time zero, rather than $u(t,{\bf x})$ at the general time $t\in [0,T]$.

\subsection{Discretized recursion without nested conditional expectations}
\label{subsection discretized recursion}

Here, we construct a forward scheme by discretizing the forward process in the probabilistic representations \eqref{w_0_E} and \eqref{w_m_E}, thereby avoiding nested conditional expectations.
To this end, fix $n\in \mathbb{N}$ and ${\bf x}\in \mathbb{R}^{d_{\bf x}}$ as appropriate, and construct a discretization $\{\overline{X}_{t_k}^{(n)}\}_{k\in \{0,1,\cdots,n\}}$ of the forward process without jumps $\{\widetilde{X}_s^{0,{\bf x}}:\,s\in [0,T]\}$ formulated in \eqref{FBSDE_without_jump} on the equidistant time points $0=:t_0<t_1<\cdots<t_n:=T$ with $t_{k+1}-t_k\equiv T/n=:\Delta_{n}$, defined recursively as follows:
\begin{equation}\label{discretized forward process}
\overline{X}_{t_{k+1}}^{(n)}:=\overline{X}_{t_k}^{(n)}+\left[b(t_k,\overline{X}_{t_k}^{(n)})-\int_{\mathbb{R}^{d_{\bf z}}_0}h(t_k,\overline{X}_{t_k}^{(n)},{\bf z})\lambda(t_k,\overline{X}_{t_k}^{(n)})v(d{\bf z})\right]\Delta_n+\sigma(t_k,\overline{X}_{t_k}^{(n)})(B_{t_{k+1}}-B_{t_k}),
\end{equation}
for all $k\in \{0,1,\cdots,n-1\}$ with $\overline{X}_{t_0}^{(n)}:={\bf x}$.
Then, starting with ${\cal U}_{m,n}({\bf x})=\Phi({\bf x})$ and ${\cal V}_{m,n}({\bf x})=(\nabla \Phi \sigma)(t_n,{\bf x})$ for $m\in \mathbb{N}_0$, we wish to define functions ${\cal U}_{m,k}:\mathbb{R}^{d_{\bf x}}\rightarrow\mathbb{R}^{d_{\bf y}}$ and ${\cal V}_{m,k}:\mathbb{R}^{d_{\bf x}}\rightarrow\mathbb{R}^{d_{\bf y}\times d_0}$
for $(m,k)\in\mathbb{N}_0\times\{0,1,\cdots,n-1\}$, satisfying the following recursion in terms of conditional expectations:
\begin{equation}
{\cal U}_{0,k}(\overline{X}_{t_k}^{(n)})=\mathbb{E}\left[\Phi(\overline{X}_{t_{n}}^{(n)}) \,\big|{\cal F}_{t_k}\right],\quad
{\cal V}_{0,k}(\overline{X}_{t_k}^{(n)})=
\mathbb{E}\left[\Phi(\overline{X}_{t_{n}}^{(n)}) \frac{(B_{t_{k+1}}-B_{t_k})^{\top}}{\Delta_{n}}\,\Big|{\cal F}_{t_k}\right],\label{w0_E_n_m_steps_min}
\end{equation}
and then recursively for $m\in\mathbb{N}$,
\begin{equation}
{\cal U}_{m,k}(\overline{X}_{t_k}^{(n)})=\mathbb{E}\left[ \Phi(\overline{X}_{t_{n}}^{(n)}) +\sum_{i=k+1}^n\Delta_{n}f({\cal U}_{m-1,i}(\overline{X}_{t_i}^{(n)}),{\cal V}_{m-1,i}(\overline{X}_{t_i}^{(n)}),\gamma(t_i,\overline{X}_{t_i}^{(n)};{\cal U}_{m-1,i}))+\sum_{i=k+1}^n\Delta_{n}\gamma(t_i,\overline{X}_{t_i}^{(n)};{\cal U}_{m-1,i})
\,\Big|{\cal F}_{t_k}\right],\label{wm_E_n_m_steps_min}
\end{equation}
and
\begin{multline}
{\cal V}_{m,k}(\overline{X}_{t_k}^{(n)})=\mathbb{E}\Bigg[\Bigg(\Phi(\overline{X}_{t_{n}}^{(n)}) +\sum_{i=k+1}^n\Delta_{n}f({\cal U}_{m-1,i}(\overline{X}_{t_i}^{(n)}),{\cal V}_{m-1,i}(\overline{X}_{t_i}^{(n)}),\gamma(t_i,\overline{X}_{t_i}^{(n)};{\cal U}_{m-1,i}))\\
+\sum_{i=k+1}^n\Delta_{n}\gamma(t_i,\overline{X}_{t_i}^{(n)};{\cal U}_{m-1,i})\Bigg) \frac{(B_{t_{k+1}}-B_{t_k})^{\top}}{\Delta_{n}}\,\Big|{\cal F}_{t_k}\Bigg],\label{vm_E_n_m_steps_min}
\end{multline}
where ${\cal U}_{m,0}({\bf x})$ can be returned as an approximation of the quantity of interest $w_m(0,{\bf x})$.

Now that the discretized recursion \eqref{discretized forward process}-\eqref{vm_E_n_m_steps_min} avoids nested conditional expectations, it can be interpreted as a forward scheme \cite[Section 5]{10.1214/23-PS18}, provided that it is implementable. 
For illustrative purposes, we outline its flow in Figure \ref{Fig_Algo_Fwd_full}, where we denote by $\overline{\cal U}_{m,k}$ an approximation of the conditional expectation ${\cal U}_{m,k}$ when employing a suitable machine learning based estimator, which we construct shortly in Sections \ref{subsection LSMC} and \ref{subsection neural network}.

\begin{figure}[H]
\centering
\includegraphics[width=15cm]{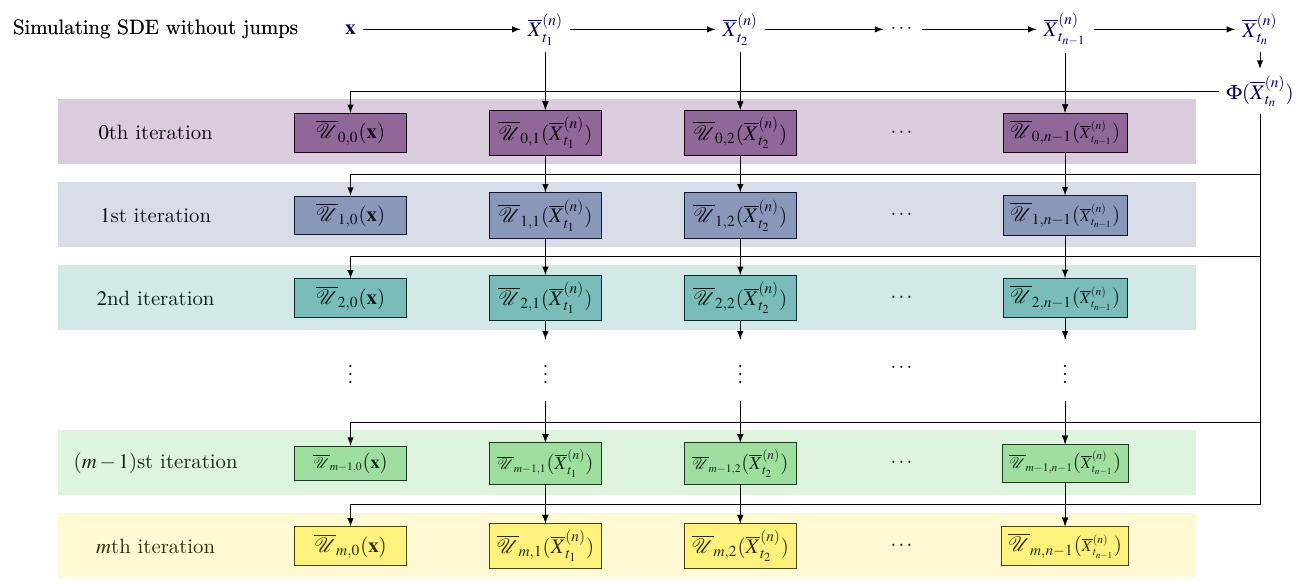}
\caption{Flow of the proposed forward scheme}
\label{Fig_Algo_Fwd_full}
\end{figure}

In a similar spirit to the forward scheme proposed in \cite{Picard_it1} (as outlined in Remark \ref{remark bender denk 2}), the advantage of the developed forward scheme is easily justifiable by the uniformly bounded approximation error, no matter how many subintervals we set up.
This feature stands in sharp contrast to typical backward schemes \cite[Section 3]{10.1214/23-PS18}, which we outline, as follows.

\noindent {\bf Forward scheme:}
In the absence of nested conditional expectations (as illustrated in Figure \ref{Fig_Algo_Fwd_loc}), the approximation error of the proposed forward scheme can be estimated as
\begin{equation}\label{do not propagate}
  \sup_{k\in \{0,1,\cdots,n\}}\mathbb{E}\left[\left|{\cal U}_{m,k}(\overline{X}_{t_k}^{(n)})-\overline{\cal U}_{m,k}(\overline{X}_{t_k}^{(n)})\right|^2\right]\leq C,\quad n\in \mathbb{N},
\end{equation}
for some positive constant $C$.
Note that in Figure \ref{Fig_Algo_Fwd_loc}, we have denoted by 
$\widehat{\mathbb{E}}[\cdot|{\cal F}_{t_k}]$ a suitable approximation of the conditional expectation $\mathbb{E}[\cdot|{\cal F}_{t_k}]$ appearing in the discretized recursion \eqref{w0_E_n_m_steps_min}-\eqref{vm_E_n_m_steps_min}.

\begin{figure}[H]
\centering
\includegraphics[width=15cm]{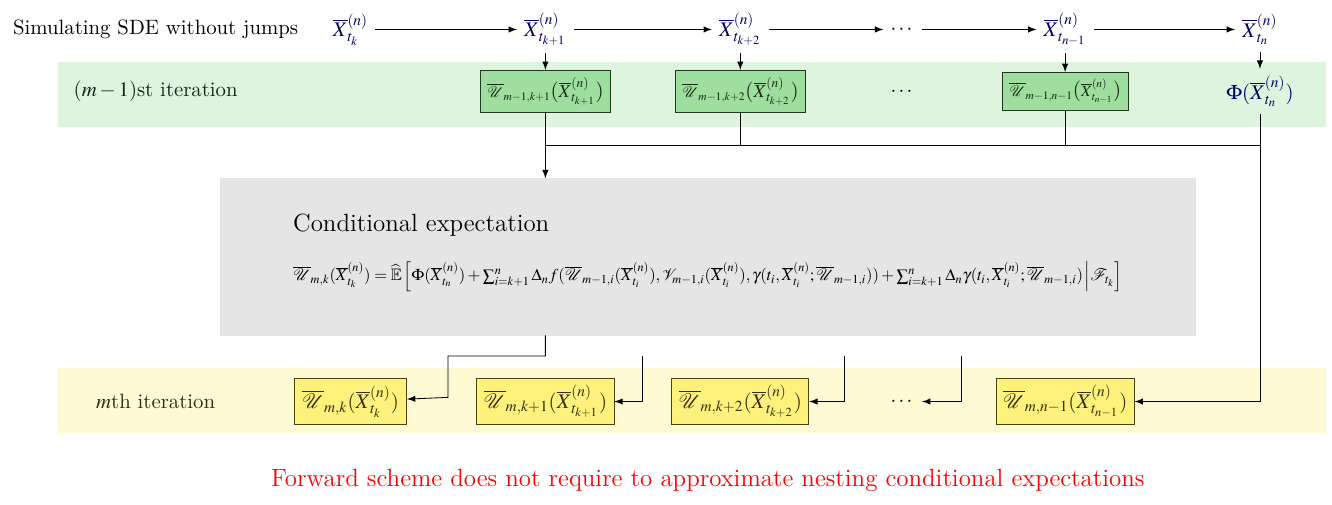}
\caption{Illustration of the absence of nested conditional expectations in the proposed forward scheme}
\label{Fig_Algo_Fwd_loc}
\end{figure}


\noindent
{\bf Backward scheme:}
In the presence of nested conditional expectations (as indicated in Figure \ref{Fig_Algo_Bwd_loc}), the approximation error is typically estimated as
\begin{equation}\label{explodes along subintervals}
  \sup_{k\in \{0,1,\cdots,n\}}\mathbb{E}\left[\left|{\cal U}_{\infty,k}(\overline{X}_{t_k}^{(n)})-\overline{\cal U}^{{\rm Euler}}_{\infty,k}(\overline{X}_{t_k}^{(n)})\right|^2\right]\leq Cn,\quad n\in\mathbb{N},
\end{equation}
for some positive constant $C$, where we have broadly followed the notation adopted in \cite[Section 3]{Picard_it1}, that is, ${\cal U}_{\infty,k}$ corresponds to the conditional expectation defined recursively in a backward manner as
\begin{equation}
{\cal U}_{\infty,k}(\overline{X}_{t_k}^{(n)}):=\mathbb{E}\left[ {\cal U}_{\infty,k+1}(\overline{X}_{t_{k+1}}^{(n)}) +\Delta_{n}f({\cal U}_{\infty,k+1}(\overline{X}_{t_{k+1}}^{(n)}),{\cal V}_{\infty,k+1}(\overline{X}_{t_{k+1}}^{(n)}),\gamma(t_{k+1},\overline{X}_{t_{k+1}}^{(n)};{\cal U}_{\infty,{k+1}}))+\Delta_{n}\gamma(t_{k+1},\overline{X}_{t_{k+1}}^{(n)};{\cal U}_{\infty,{k+1}})
\,\Big|{\cal F}_{t_k}\right],\label{u_inf}
\end{equation}
and
\begin{multline}
{\cal V}_{\infty,k}(\overline{X}_{t_k}^{(n)}):=\mathbb{E}\Bigg[\Bigg( {\cal U}_{\infty,k+1}(\overline{X}_{t_{k+1}}^{(n)})+\Delta_{n}f({\cal U}_{\infty,k+1}(\overline{X}_{t_{k+1}}^{(n)}),{\cal V}_{\infty,k+1}(\overline{X}_{t_{k+1}}^{(n)}),\gamma(t_{k+1},\overline{X}_{t_{k+1}}^{(n)};{\cal U}_{\infty,k+1}))\\
+\Delta_{n}\gamma(t_{k+1},\overline{X}_{t_{k+1}}^{(n)};{\cal U}_{\infty,k+1})\Bigg) \frac{(B_{t_{k+1}}-B_{t_k})^{\top}}{\Delta_{n}}\,\Big|{\cal F}_{t_k}\Bigg],\label{v_inf}
\end{multline}
for $k\in\{n-1,\cdots,0,1\}$, while $\overline{\cal U}_{\infty, k}^{\rm Euler}$ represents an approximation of the conditional expectation ${\cal U}_{\infty,k}$, with $\widehat{\mathbb{E}}[\cdot|{\cal F}_{t_k}]$ in Figure \ref{Fig_Algo_Bwd_loc} approximating $\mathbb{E}[\cdot|{\cal F}_{t_k}]$ in \eqref{u_inf}-\eqref{v_inf}.
\begin{figure}[H]
\centering
\includegraphics[width=15cm]{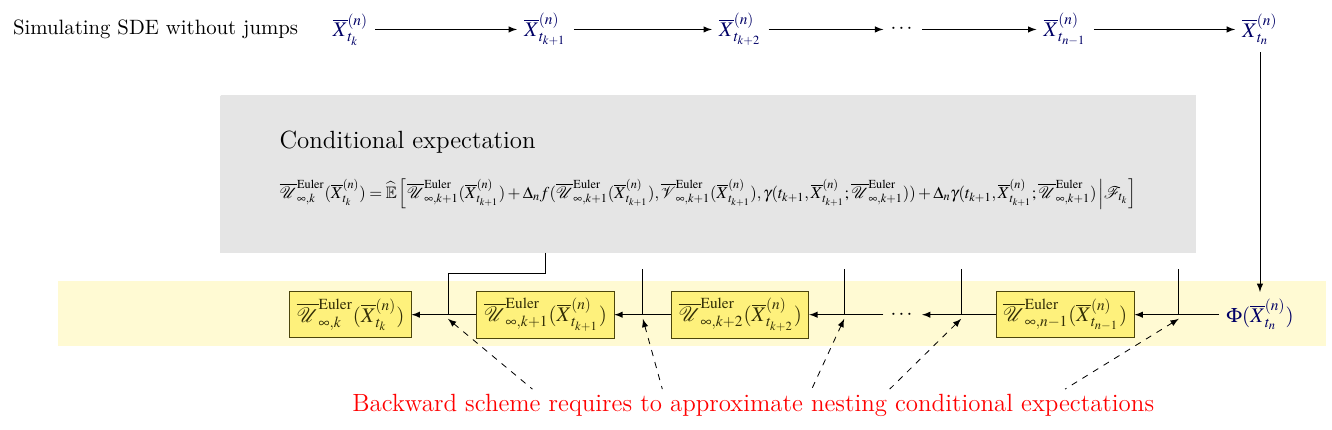}
\caption{Illustration of the presence of nested conditional expectations in a typical backward Euler scheme}
\label{Fig_Algo_Bwd_loc}
\end{figure}

In summary, the approximation error of the forward scheme does not get propagated along subintervals (as indicated through a constant upper bound in \eqref{do not propagate}), whereas in general, the approximation error of the backward scheme itself, not merely its upper bound, explodes as the number of subintervals grows (as implied in \eqref{explodes along subintervals}).
To avoid digressing too much from the main topic, we do not go into further details here, but refer the reader to \cite[Section 3]{Picard_it1}.

In the next two subsections, we describe two machine learning techniques for implementing the discretized recursion \eqref{w0_E_n_m_steps_min}-\eqref{vm_E_n_m_steps_min} in the form of minimization problems for approximating the conditional expectations involved.
In what follows, we indicate by $a\in \argmin$ that $a$ is one of the arguments that minimize a given function.

\subsection{Implementation by least squares Monte Carlo}
\label{subsection LSMC}

Here, we describe the least squares Monte Carlo method for implementing the forward scheme \eqref{w0_E_n_m_steps_min}-\eqref{vm_E_n_m_steps_min} (for instance, \cite{Picard_it1} and \cite[Section 4.1]{10.1214/23-PS18} for details). 
With $J\in\mathbb{N}$ fixed, let $\{v_j^{\mathrm{basis}}\}_{j\in \{1,2,\cdots,J\}}$ be a sequence of basis functions from $\mathbb{R}^{d_{\bf x}}$ to $\mathbb{R}$. 
Starting with ${\cal U}_{m,n}^{\mathrm{LS}}({\bf x})=\Phi({\bf x})$ and ${\cal V}_{m,n}^{\rm LS}({\bf x})=(\nabla \Phi \sigma)(t_n,{\bf x})$ for $m\in\mathbb{N}_0$, we approximate the condition expectations appearing in the recursion \eqref{w0_E_n_m_steps_min}-\eqref{vm_E_n_m_steps_min} with a linear combination of basis functions evaluated at $\overline{X}^{(n)}_{t_k}$, as follows:
for $k\in \{n-1,\cdots,0\}$,
\begin{gather}
\{\beta_{0,k,j}^{{\cal U}}\}_{j\in \{1,\cdots,J\}}\in \argmin_{\{\beta_{j}\}_{j\in \{1,\cdots,J\}}\subset \mathbb{R}^{d_{\bf y}}}\mathbb{E}\left[\left|\sum\nolimits_{j=1}^J\beta_jv_j^{\mathrm{basis}}(\overline{X}_{t_k}^{(n)})- \Phi(\overline{X}_{t_n}^{(n)})\right|^2\right],\label{w0_E_n_LS}\\
\{\beta_{0,k,j}^{{\cal V}}\}_{j\in \{1,\cdots,J\}}\in \argmin_{\{\beta_{j}\}_{j\in \{1,\cdots,J\}}\subset\mathbb{R}^{d_{\bf y}\times d_0}}\mathbb{E}\left[\left|\sum\nolimits_{j=1}^J\beta_jv_j^{\mathrm{basis}}(\overline{X}_{t_k}^{(n)})-\Phi(\overline{X}_{t_{n}}^{(n)}) \frac{(B_{t_{k+1}}-B_{t_k})^{\top}}{\Delta_{n}}\right|^2\right],\label{v0_E_n_LS}\\
{\cal U}_{0,k}^{\mathrm{LS}}(\overline{X}_{t_k}^{(n)})=\sum_{j=1}^J\beta_{0,k,j}^{{\cal U}}v_j^{\mathrm{basis}}(\overline{X}_{t_k}^{(n)}),\quad 
{\cal V}_{0,k}^{\mathrm{LS}}(\overline{X}_{t_k}^{(n)})=\sum_{j=1}^J\beta_{0,k,j}^{{\cal V}}v_j^{\mathrm{basis}}(\overline{X}_{t_k}^{(n)}),\nonumber
\end{gather}
and, then recursively for $m\in\mathbb{N}$ and $k\in \{n-1,\cdots,0\}$,
\begin{align}
&\{\beta_{m,k,j}^{{\cal U}}\}_{j\in \{1,\cdots,J\}}\in \argmin_{\{\beta_{j}\}_{j\in \{1,\cdots,J\}}\subset\mathbb{R}^{d_{\bf y}}}\mathbb{E}\bigg[\bigg|\sum_{j=1}^J\beta_{j}v_j^{\mathrm{basis}}(\overline{X}_{t_k}^{(n)})\nonumber\\
&\qquad
-\bigg(\Phi(\overline{X}_{t_{n}}^{(n)}) +\sum_{i=k+1}^{n}\Delta_{n}f({\cal U}_{m-1,i}^{\mathrm{LS}}(\overline{X}_{t_{i}}^{(n)}),{\cal V}_{m-1,i}^{\mathrm{LS}}(\overline{X}_{t_{i}}^{(n)}),\gamma(t_{i},\overline{X}_{t_{i}}^{(n)};{\cal U}_{m-1,i}^{\mathrm{LS}}))
+\sum_{i=k+1}^{n}\Delta_{n}\gamma(t_{i},\overline{X}_{t_{i}}^{(n)};{\cal U}_{m-1,i}^{\mathrm{LS}})
\bigg)\bigg|^2\bigg],\label{wm_E_n_LS}\\
&\{\beta_{m,k,j}^{{\cal V}}\}_{j\in \{1,\cdots,J\}}\in \argmin_{\{\beta_{j}\}_{j\in \{1,\cdots,J\}}\subset\mathbb{R}^{d_{\bf y}\times d_0}}\mathbb{E}\bigg[\bigg|\sum_{j=1}^J\beta_jv_{j}^{\mathrm{basis}}(\overline{X}_{t_k}^{(n)})\nonumber\\
&\qquad -\bigg( \Phi(\overline{X}_{t_{n}}^{(n)}) +\sum_{i=k+1}^{n}\Delta_{n}f({\cal U}_{m-1,i}^{\mathrm{LS}}(\overline{X}_{t_{i}}^{(n)}),{\cal V}_{m-1,i}^{\mathrm{LS}}(\overline{X}_{t_{i}}^{(n)}),\gamma(t_{i},\overline{X}_{t_{i}}^{(n)};{\cal U}_{m-1,i}^{\mathrm{LS}}))+\sum_{i=k+1}^{n}\Delta_{n}\gamma(t_{i},\overline{X}_{t_{i}}^{(n)};{\cal U}_{m-1,i}^{\mathrm{LS}})
\bigg)\frac{(B_{t_{k+1}}-B_{t_k})^{\top}}{\Delta_{n}}\bigg|^2\bigg],\label{vm_E_n_LS}\\
&{\cal U}_{m,k}^{\mathrm{LS}}(\overline{X}_{t_k}^{(n)})=\sum_{j=1}^J\beta_{m,k,j}^{{\cal U}}v_j^{\mathrm{basis}}(\overline{X}_{t_k}^{(n)}),\quad {\cal V}_{m,k}^{\mathrm{LS}}(\overline{X}_{t_k}^{(n)})=\sum_{j=1}^J\beta_{m,k,j}^{{\cal V}}v_j^{\mathrm{basis}}(\overline{X}_{t_k}^{(n)}).\nonumber
\end{align}
Then, we return the value ${\cal U}_{m,0}^{\mathrm{LS}}({\bf x})$ as an approximation of $w_m(0,{\bf x})$.
For the sake of convenience in coding, we provide a pseudocode for the least squares Monte Carlo method above as Algorithm \ref{Algo_LSMC} in the supplementary materials.

\subsection{Implementation by neural networks}
\label{subsection neural network}

For exceptionally large problem dimensions, especially $d_{\bf x}$, the least squares Monte Carlo method of Section \ref{subsection LSMC} becomes infeasible due to its substantial memory requirements.
For such high-dimensional problems, it is often effective to employ neural networks (for instance, \cite{doi:10.1137/19M1297919} and \cite[Section 6]{10.1214/23-PS18} for details) for implementing the recursion \eqref{w0_E_n_m_steps_min}-\eqref{vm_E_n_m_steps_min}.

Here, we introduce a neural network architecture of the scheme.
We define a Rectified Linear Unit (ReLU) activation function by ${\cal L}_{r}({\bf x}):=( (x_1)^+, \cdots, (x_r)^+)^{\top}$ for ${\bf x}\in \mathbb{R}^r$. 
With $\nu\in\mathbb{N}$ fixed, let  $p,\ell \in \mathbb{N}$, $q \in \mathbb{N}_0$, and $\theta=(\theta_{1},\cdots,\theta_{\nu}) \in \mathbb{R}^\nu$ such that $q+\ell p+\ell  \leq \nu$.
We define 
\[
A^{\theta,q}_{p,\ell}({\bf x})
:=\begin{bmatrix} 
\theta_{q+1} & \cdots & \theta_{q+p} \\
\vdots & \ddots & \vdots \\
\theta_{q+(\ell-1)p+1} & \cdots & \theta_{q+(\ell-1) p+p} 
\end{bmatrix}
{\bf x}+ 
\begin{bmatrix}
\theta_{q+\ell p+1}  \\
\vdots  \\
\theta_{q+\ell p+\ell}  
\end{bmatrix},\quad {\bf x}=\begin{bmatrix} 
x_1  \\
\vdots  \\
x_p  
\end{bmatrix}.
\]
Note that we have above prepared the parameter $\theta$ with $v$ components, instead of 
the minimal necessary $(\ell p+\ell)$ components of $(\theta_{q+1},\ldots,\theta_{q+\ell p+\ell})$, 
for simplicity of notation.

Let $s\in\{3,4,\cdots\}$ denote the number of layers (excluding the input layer) and let $d_k\in\mathbb{N}$ represent the number of neurons in the $k$ layer for $k\in \{0,1,\cdots,s\}$, satisfying $\sum_{k=1}^{s}d_k(d_{k-1}+1)\leq \nu$.
With those, we define a function ${\rm NN}:\mathbb{R}^{d_0}\rightarrow\mathbb{R}^{d_s}$ with the parameter $\theta\in\mathbb{R}^\nu$ by
\[
{\rm NN}({\bf x};\theta)
:=\bigg(A_{d_{s-1},d_s}^{\theta,\sum_{k=1}^{s-1}d_{k}(d_{k-1}+1)} \circ {\cal L}_{d_{s-1}} \circ A_{d_{s-2},d_{s-1}}^{\theta,\sum_{k=1}^{s-2}d_{k}(d_{k-1}+1)} \circ \cdots \circ {\cal L}_{d_2} \circ A_{d_1,d_2}^{\theta,d_1(d_0+1)} \circ 
{\cal L}_{d_1} \circ A_{d_0,d_1}^{\theta,0}\bigg)({\bf x}),\ \ \ \ {\bf x}\in\mathbb{R}^{d_0},
\]
that we call the neural network with the parameter $\theta$, composed of $(s+1)$ layers ($1$ input layer with $d_0$ neurons, $k$ hidden layers with $d_k$ neurons for each $k\in \{1,\cdots,s-1\}$ and $1$ output layer with $d_s$ neurons).
In the present context, we prepare two such neural networks, one for the components involving $\mathcal{U}$ appearing in \eqref{w0_E_n_m_steps_min} and \eqref{wm_E_n_m_steps_min}, while the other for $\mathcal{V}$ in \eqref{w0_E_n_m_steps_min} and \eqref{vm_E_n_m_steps_min}.
For the former, we have ${\rm NN}^{{\cal U}}(\cdot;\theta)$ with $d^{\cal U}_0=d_{\bf x}$, $d^{\cal U}_{s}=d_{\bf y}$, and $d^{\cal U}_k$ for $k\in\{n-1,\cdots,0\}$ satisfying $\sum_{k=1}^{s}d_k^{\cal U}(d_{k-1}^{\cal U}+1)\leq \nu$.
In a similar manner, for the latter, we have ${\rm NN}^{{\cal V}}(\cdot;\theta)$ with $d^{\cal V}_0=d_{\bf x}$, $d^{\cal V}_s=d_{\bf y}\times d_0$, and $d^{\cal V}_k\in \mathbb{N}$ for $k\in \{1,\cdots,s-1\}$, satisfying $\sum_{k=1}^{s}d_k^{\cal V}(d_{k-1}^{\cal V}+1)\leq \nu$.

With those neural network structures, we set up a forward scheme with ${\cal U}_{m,n}^{\mathrm{NN}}({\bf x})=\Phi({\bf x})$ and ${\cal V}_{m,n}^{\mathrm{NN}}({\bf x})=(\nabla \Phi \sigma)(t_n,{\bf x})$ for $m\in\mathbb{N}_0$, as follows: for $k\in \{0,\cdots,n-1\}$,
\begin{gather}
\Theta^{{\cal U}}_{0,k}\in \argmin_{\theta\in\mathbb{R}^\nu}\mathbb{E}\left[\left|{\rm NN}^{{\cal U}}(\overline{X}_{t_k}^{(n)};\theta)- \Phi(\overline{X}_{t_{n}}^{(n)}) \right|^2\right],\label{w0_E_n_NN}\\
\Theta^{{\cal V}}_{0,k}\in \argmin_{\theta\in\mathbb{R}^{\nu}}\mathbb{E}\left[\left|{\rm NN}^{{\cal V}}(\overline{X}_{t_k}^{(n)};\theta)- \Phi(\overline{X}_{t_{n}}^{(n)})\frac{(B_{t_{k+1}}-B_{t_k})^{\top}}{\Delta_{n}}\right|^2\right],\label{v0_E_n_NN}\\
{\cal U}_{0,k}^{\mathrm{NN}}(\overline{X}_{t_k}^{(n)})={\rm NN}^{{\cal U}}(\overline{X}_{t_k}^{(n)};\Theta_{0,k}^{{\cal U}}),\quad {\cal V}_{0,k}^{\mathrm{NN}}(\overline{X}_{t_k}^{(n)})={\rm NN}^{{\cal V}}(\overline{X}_{t_k}^{(n)};\Theta_{0,k}^{{\cal V}}),\nonumber
\end{gather}
and then, recursively for $m\in\mathbb{N}$ and $k\in \{n-1,\cdots,0\}$,
\begin{align}
&\Theta^{{\cal U}}_{m,k}\in \argmin_{\theta\in\mathbb{R}^\nu}\mathbb{E}\bigg[\bigg|{\rm NN}^{{\cal U}}(\overline{X}_{t_k}^{(n)};\theta)-\bigg(\Phi(\overline{X}_{t_{n}}^{(n)}) +\sum_{i=k+1}^{n}\Delta_{n}f({\cal U}_{m-1,i}^{\mathrm{NN}}(\overline{X}_{t_{i}}^{(n)}),{\cal V}_{m-1,i}^{\mathrm{NN}}(\overline{X}_{t_{i}}^{(n)}),\gamma(t_{i},\overline{X}_{t_{i}}^{(n)};{\cal U}_{m-1,i}^{\mathrm{NN}}))\nonumber\\
&\qquad \qquad +\sum_{i=k+1}^{n}\Delta_{n}\gamma(t_{i},\overline{X}_{t_{i}}^{(n)};{\cal U}_{m-1,i}^{\mathrm{NN}})
\bigg)\bigg|^2\bigg],\label{wm_E_n_NN}\\
&\Theta^{{\cal V}}_{m,k}\in \argmin_{\theta\in\mathbb{R}^{\nu}}
\mathbb{E}\bigg[\bigg|{\rm NN}^{{\cal V}}(\overline{X}_{t_k}^{(n)};\theta)-\bigg(\Phi(\overline{X}_{t_{n}}^{(n)}) +\sum_{i=k+1}^{n}\Delta_{n}f({\cal U}_{m-1,i}^{\mathrm{NN}}(\overline{X}_{t_{i}}^{(n)}),{\cal V}_{m-1,i}^{\mathrm{NN}}(\overline{X}_{t_{i}}^{(n)}),\gamma(t_{i},\overline{X}_{t_{i}}^{(n)};{\cal U}_{m-1,i}^{\mathrm{NN}}))\nonumber\\
&\qquad \qquad +\sum_{i=k+1}^{n}\Delta_{n}\gamma(t_{i},\overline{X}_{t_{i}}^{(n)};{\cal U}_{m-1,i}^{\mathrm{NN}})
\bigg)\frac{(B_{t_{k+1}}-B_{t_k})^{\top}}{\Delta_{n}}\bigg|^2\bigg],\label{vm_E_n_NN}\\
&{\cal U}_{m,k}^{\mathrm{NN}}(\overline{X}_{t_k}^{(n)})={\rm NN}^{{\cal U}}(\overline{X}_{t_k}^{(n)};\Theta_{m,k}^{{\cal U}}),\quad {\cal V}_{m,k}^{\mathrm{NN}}(\overline{X}_{t_k}^{(n)})={\rm NN}^{{\cal V}}(\overline{X}_{t_k}^{(n)};\Theta_{m,k}^{{\cal V}}).\nonumber
\end{align}
Then, we return the value ${\cal U}_{m,0}^{\mathrm{NN}}({\bf x})$ as an approximation of $w_m(0,{\bf x})$.
For the sake of convenience in coding, we provide a pseudocode for the neural network-based scheme above as Algorithm \ref{Algo_NN} in the supplementary materials.

\begin{remark}\label{remark bender denk 3}{\rm
It is worth noting, in continuation of Remark \ref{remark bender denk 2}, that the developed neural network-based scheme \eqref{w0_E_n_NN}-\eqref{vm_E_n_NN} can also accommodate the implementation of the existing forward method for standard FBSDEs without jumps \cite{Picard_it1}.
Given this capability, the above machine learning-based implementation introduces extra novelty by handling a broad range of forward schemes for FBSDEs, covering not only those with jumps but notably also those without jumps.
\qed}\end{remark}

Before closing this section, let us contrast the two frameworks from a practical perspective.
As mentioned, on the one hand, the least squares Monte Carlo is effective in low dimensional problems.
It however becomes prohibitive for high-dimensional cases (roughly, $d_{\bf x}>15$ in the present context), due to its substantial memory requirements.
On the other hand, the neural network-based scheme is not always advantageous, either.
Specifically, it often does not yield benefits for low-dimensional problems due to the significant time required for the initial setup of the neural networks, no matter how small the problem dimension is.
Hence, in the following numerical illustrations, we exclusively employ the former for low-dimensional cases and the latter for high-dimensional ones.



\section{Numerical illustrations}
\label{section numerical illustrations}

To verify the validity of the recursive representation (Section \ref{section recursive representation}) and demonstrate the effectiveness of the developed forward scheme and its machine learning-based implementation (Section \ref{section forward scheme}), particularly in high-dimensional settings, we provide a diverse yet specifically targeted range of problem scenarios.
Without loss of generality, we continue to focus on the approximation of the solution at time zero (that is, $Y_0^{0,{\bf x}}=u(0,{\bf x})$ in accordance with the identity \eqref{nFK_jump}).
All numerical experiments below are conducted in Python on two GPUs of NVIDIA RTX A6000 48GB with NVLink, on a system with an AMD EPYC 7402P 2.80GHz 24-core CPU with 128GB DDR4-3200 memory running on Ubuntu 23.10.

To ease the notation, we write ${\bf x}=(x_1,\cdots,x_d)^{\top}$ and ${\bf z}=(z_1,\cdots,z_d)^{\top}$, depending on the their dimensions $d$. 
Also, we write the average of a vector in a matrix operation form, for instance, $d^{-1}\sum_{k=1}^d x_k=\langle {\bf x},\mathbbm{1}_d/d\rangle$.
Recall that $\mathbb{I}_d$ and $\mathbbm{1}_d$ represent the identity matrix of order $d$ and the $d$-dimensional vector with all unit-valued components, respectively.

\subsection{First example}
\label{subsection first example}


First, we examine the recursive representation (Section \ref{section recursive representation}) in the context of the Merton jump-diffusion model \cite{merton1976option} capitalizing its semi-analytic formula (to be presented shortly in \eqref{Merton_wm}), so that the experiment can be conducted without worrying about significant numerical errors.

In alignment with the formulation \eqref{FBSDE_with_jump}, we set $d_{\bf x}=d_0=d_{\bf z}=d$ for some $d\in \mathbb{N}$, $d_{\bf y}=1$, 
$h({\bf x},{\bf z}):=(x_1(e^{z_1}-1),\cdots,x_d(e^{z_d}-1))^{\top}$, $\lambda(s,{\bf x})\equiv \lambda$ for some $\lambda>0$, $\nu(d{\bf z})=(2\pi \det(\Sigma_J))^{-d/2}e^{-\langle {\bf z}-\mu_J, \Sigma_J^{-1}({\bf z}-\mu_J)\rangle/2}d{\bf z}$ with $\mu_J=(\mu_J^1,\cdots,\mu_J^d)^{\top}\in\mathbb{R}^d$ and $\sigma_J=(\sigma_J^{ij})_{i,j\in \{1,\cdots,d\}}\in \mathbb{R}^{d\times d}$, such that $\Sigma_J:=\sigma_J^{\otimes 2}$ is invertible.
The model of interest is then described as follows: for $s\in [0,T]$,
\begin{equation}\label{d_dim_Merton}
\begin{dcases}
dX_s^{0,{\bf x}} = r X_s^{0,{\bf x}} ds+{\rm diag}(X_s^{0,{\bf x}})\sigma dB_s+ \int_{\mathbb{R}^d_0}h(X_{s-}^{0,{\bf x}},{\bf z})\widetilde{\mu}(d{\bf z},ds),\\
dY_s^{0,{\bf x}}=(rY_s^{0,{\bf x}}+c(Y_s^{0,{\bf x}})^+)ds+Z_s^{0,{\bf x}}dB_s+ \int_{\mathbb{R}^d_0}U_{s-}^{0,{\bf x}}({\bf z})\widetilde{\mu}(d{\bf z},ds),\\
X_0^{0,{\bf x}}= {\bf x}\in\mathbb{R}^d,\quad Y_T^{0,{\bf x}}=\Phi(X_T^{0,{\bf x}}),
\end{dcases}
\end{equation}
with $X_0^{0,{\bf x}}= {\bf x}\in\mathbb{R}^d$, $Y_T^{0,{\bf x}}=\Phi(X_T^{0,{\bf x}})$, $r>0$, $\sigma\in\mathbb{R}^{d\times d}$, $c>0$, and $\Phi({\bf x})=(K-\langle {\bf x},\mathbbm{1}_d/d\rangle)^+$ for some $K>0$, where the constant $c$ represents the default intensity of the counterparty.
This formulation corresponds to the multi-dimensional Merton jump-diffusion model \cite{doi:10.1080/00207160.2015.1070838, doi:10.1080/17442508.2010.515309}, whose jump times are governed by a single Poisson process (due to $\lambda(s,{\bf x})\equiv \lambda$) to describe a contagious jump effect on all involving assets (with possibly distinct jump sizes for each asset, governed by the finite L\'evy measure $\nu(d{\bf z})$).
In accordance with the formulation \eqref{PIDE}, the PIDE corresponding to the FBSDE with jumps \eqref{d_dim_Merton} is given as follows: for $(s,{\bf x})\in [0,T)\times \mathbb{R}^d$,
\begin{multline*}
\frac{\partial}{\partial s}u(s,{\bf x})+\langle r{\bf x},\nabla u(s,{\bf x})\rangle +\frac{1}{2}{\rm tr}\left[({\rm diag}({\bf x})\sigma)^{\otimes 2}\mathrm{Hess}(u(s,{\bf x}))\right]\\
+\int_{\mathbb{R}^d_0}(u(s,{\bf x}+h({\bf x},{\bf z}))-u(s,{\bf x})-\langle h({\bf x},{\bf z}),\nabla u(s,{\bf x})\rangle)\lambda\nu(d{\bf z})-ru(s,{\bf x})-c(u(s,{\bf x}))^+=0,
\end{multline*}
with $u(T,{\bf x})=\Phi({\bf x})$.

Next, let $b({\bf x}):=((r-\lambda (e^{\mu_J^1+(\sigma_J^{11})^2/2}-1))x_1,\cdots,(r-\lambda (e^{\mu_J^d+(\sigma_J^{dd})^2/2}-1))x_d)^{\top}$, where the terms $\lambda(e^{\mu_J^k+(\sigma_J^{kk})^2/2}-1)x_k$ come from $\int_{\mathbb{R}^d_0}h({\bf x},{\bf z})\lambda v(d{\bf z})=(\lambda (e^{\mu_J^1+(\sigma_J^{11})^2/2}-1)x_1,\cdots,\lambda(e^{\mu_J^d+(\sigma_J^{dd})^2/2}-1)x_d)^{\top}$, with which the recursion $\{w_m\}_{m\in\mathbb{N}_0}$ developed in \eqref{PIDE_0} and \eqref{PIDE_m} can be expressed in the form of linear PDEs, as follows: for $(s,{\bf x})\in [0,T)\times \mathbb{R}^d$,
\begin{equation}
\frac{\partial}{\partial s}w_0(s,{\bf x})+\langle 
b({\bf x}),\nabla w_0(s,{\bf x})\rangle+\frac{1}{2}{\rm tr}\left[({\rm diag}({\bf x})\sigma)^{\otimes 2}\mathrm{Hess}(w_0(s,{\bf x}))\right]=0,\label{PIDE_d_Merton_w0}
\end{equation}
and then, recursively for $m\in \mathbb{N}$,
\begin{multline}
\frac{\partial}{\partial s}w_m(s,{\bf x})+
\langle 
b({\bf x}),\nabla w_m(s,{\bf x})\rangle
+\frac{1}{2}{\rm tr}\left[(({\rm diag}({\bf x})\sigma)^{\otimes 2}\mathrm{Hess}(w_m(s,{\bf x}))\right]\\
+\int_{\mathbb{R}^d_0}w_{m-1}(s,{\bf x}+h({\bf x},{\bf z}))\lambda\nu(d{\bf z})-(\lambda+r)w_{m-1}(s,{\bf x})-c(w_{m-1}(s,{\bf x}))^+=0,\label{PIDE_d_Merton_wm}
\end{multline}
along with $w_m(T,{\bf x})=\Phi({\bf x})$ for all $m\in \mathbb{N}_0$.
In light of the FBSDE without jumps \eqref{FBSDE_without_jump} and Theorem \ref{theorem section 3.2}, the linear PDEs \eqref{PIDE_d_Merton_w0} and \eqref{PIDE_d_Merton_wm} admit the following probabilistic representations:
\begin{equation}
    w_0(0,{\bf x})=\mathbb{E}\left[\Phi(\widetilde{X}_T^{0,{\bf x}} )\right],
\label{Picard_itr_for_d_CVA1}
\end{equation}
where $d\widetilde{X}_s^{0,{\bf x}} = b(\widetilde{X}_s^{0,{\bf x}} )ds + {\rm diag}(\widetilde{X}_s^{0,{\bf x}} )\sigma dB_s$ with $\widetilde{X}_0^{0,{\bf x}} ={\bf x}\in\mathbb{R}^d$, and then for $m\in\mathbb{N}$,
\begin{align}
w_m(0,{\bf x})
&=\sum_{k=0}^m\mathbb{E}\left[\Phi(\widetilde{X}_T^{0,{\bf x}+h({\bf x},Z_1+\cdots+Z_k)})\right] \frac{(\lambda T)^k}{k!}\sum_{n=0}^{m-k}(c+r+\lambda)^n\frac{(-T)^n}{n!},\nonumber\\
&=\sum_{k=0}^m\mathbb{E}\left[w_0(0,{\bf x}+h({\bf x},Z_1+\cdots+Z_k))\right] \frac{(\lambda T)^k}{k!}\sum_{n=0}^{m-k}(c+r+\lambda)^n\frac{(-T)^n}{n!},
\label{Merton_wm}
\end{align}
where $\{Z_k\}_{k\in \mathbb{N}}$ here is a sequence of iid normal random vectors with mean $\mu_J$ and covariance matrix $\Sigma_J$.

It is worth noting that the representation \eqref{Merton_wm} offers an interesting interpretation on its own. 
On the one hand, in the one-dimensional case and in the absence of nonlinearity in the driver (that is, $c=0$), the solution (that is, $u(0,x)=e^{-rT}\mathbb{E}[\Phi(X_T^{0,x})]$) corresponds to the 
European price (which we denote by $\mathrm{Merton}$ below), which can be written as the sum of the plain-vanilla prices on the standard Black-Scholes model (say, $\mathrm{BS}$), as follows \cite{merton1976option}:
\begin{equation}\label{Merton_formula}
\mathrm{Merton}(x,K,T,r,\sigma,\lambda,\mu_J,\sigma_J)=\sum_{k=0}^{+\infty}\mathrm{BS}(x, K, T, r_k, \sigma_k)e^{-e^{\mu_J+\sigma_J^2/2}\lambda T}\frac{(e^{\mu_J+\sigma_J^2/2}\lambda T)^k}{k!},
\end{equation}
with $r_k=r_k(T,r,\lambda,\mu_J,\sigma_J):= r - \lambda(e^{\mu_J+\sigma_J^2/2}-1) + k(\mu_J+\sigma_J^2/2)/ T$ and $\sigma_k=\sigma_k(T,\sigma,\sigma_J):= (\sigma^2 + k\sigma_J^2/ T)^{1/2}$ for $k\in\mathbb{N}_0$.
If multi-dimensional or in the presence of nonlinearity in the driver (that is, $c>0$), on the other hand, there does not appear to be any (semi-)analytic expression available for the solution $u(0,{\bf x})$.
Therefore, the representation \eqref{Merton_wm} can be thought of as a semi-analytic formula for the European price under the multi-dimensional Merton-jump diffusion model with the nonlinear driver, where jumps are fully decoupled at each and every step.

In the numerical experiments below, we present the reference value $u_{\rm ref}(0,{\bf x})$ on the basis of following representation:
\begin{equation}
u(0,{\bf x})
=\mathbb{E}\left[\Phi(X_T^{0,{\bf x}})-\int_0^T(r+c)u(s,X_s^{0,{\bf x}})ds\right]=e^{-(r+c)T}\mathbb{E}\left[\Phi(X_T^{0,{\bf x}})\right]\label{CVA_ref_E},
\end{equation}
where the driver term $ru+c(u)_+$ has been reduced to $(r+c)u$, due to $u(s,{\bf x})\ge 0$ for all $(s,{\bf x})\in[0,T]\times\mathbb{R}^{d}$.
If one-dimensional, then the representation \eqref{CVA_ref_E} can be simplified with the aid of the formula \eqref{Merton_formula}, as follows: 
\begin{equation}\label{CVA_ref_Merton_formula}
u(0,x)=e^{-cT}\mathrm{Merton}(x,K,T,r,\sigma,\lambda,\mu_J,\sigma_J).
\end{equation}

\subsubsection{One-dimensional case}
\label{example merton univariate}

For optimal illustration, we start with the one-dimensional setting ($d=1$), with the problem parameters set as $T=1.0$, $r=0.04$, $\sigma=0.25$, $\lambda=0.5$, $\mu_J=0.5$, $\sigma_J=0.5$, and $c=0.1$.
Note that we present the variable $x$ in a non-bold font to emphasize its univariate nature.
In view of Theorem \ref{thm1}, we conduct the standard Monte Carlo estimation based on $10^7$ iid paths for accurately approximating $w_m(0,x)$ at initial state $x=10.0$ for the first five iterations $m\in \{0,1,\cdots,5\}$ and estimate the approximation errors $|u_{\rm ref}(0,x)-w_m(0,x)|$ (but without taking the supremum as in Theorem \ref{thm1}), 
where the reference values $u_{\rm ref}(0,x)$ are due to the formula \eqref{CVA_ref_Merton_formula}.
We provide in Table \ref{Tbl_Merton} numerical results at three strikes $K\in \{5.0,\,10.0,\,15.0\}$ and plot in Figure \ref{Fig_Merton} an extended set of approximation errors, here at eleven different strikes $K\in\{5.0,6.0,\cdots,15.0\}$.
Only a few iterations appear to be sufficient to achieve a relative error of $1\%$ across all presented strikes.

\begin{table}[H]
  \begin{center}\small
    \begin{tabular}{c||cc|cc|cc}
      \hline & \multicolumn{2}{c|}{$K=5.0$} & \multicolumn{2}{c|}{$K=10.0$} & \multicolumn{2}{c}{$K=15.0$}\\
      \hline
      & value & error& value & error& value & error\\\hline\hline
      $u_{\rm ref}(0,x)$ & $0.070331$ & $-$& $2.160926$ & $-$& $5.674748$ & $-$\\\hline
      $w_0(0,x)$ & $0.081661$ & $0.011330$ ($16.11\%$) & $3.307421$ & $1.146495$ ($53.06\%$)  & $8.257754$ & $2.583006$ ($45.52\%$) \\ 
      $w_1(0,x)$ & $0.073525$ & $0.003196$ ($4.54\%$) & $1.886526$ & $0.274400$ ($12.70\%$) & $5.163110$ & $0.511638$ ($9.02\%$) \\
      $w_2(0,x)$ & $0.068875$ & $0.001456$ ($2.07\%$) & $2.205167$ & $0.044241$ ($2.05\%$) & $5.738672$ & $0.063924$ ($1.13\%$) \\
      $w_3(0,x)$ & $0.070668$ & $0.000336$ ($0.48\%$) & $2.155514$ & $0.005412$ ($0.25\%$) & $5.669097$ & $0.005651$ ($0.10\%$) \\
      $w_4(0,x)$ & $0.070289$ & $0.000042$ ($0.06\%$) & $2.161511$ & $0.000585$ ($0.03\%$) & $5.675256$ & $0.000508$ ($0.01\%$) \\
      $w_5(0,x)$ & $0.070352$ & $0.000021$ ($0.03\%$) & $2.160999$ & $0.000073$ ($0.00\%$) & $5.674685$ & $0.000063$ ($0.00\%$) \\\hline
      \end{tabular}
      \caption{The reference values $u_{\rm ref}(0,x)$, the values from the recursion $\{w_m(0,x)\}_{m\in \{0,1,\cdots,5\}}$, and the approximation errors $|u_{\rm ref}(0,x)-w_m(0,x)|$ for $m\in \{0,1,\cdots,5\}$ at initial state $x=10.0$ and three different strikes $K\in \{5.0,\,10.0,\,15.0\}$, with relative errors in brackets.}
      \label{Tbl_Merton}
    \end{center}
\end{table}

\begin{figure}[H]
\centering
\includegraphics[width=9cm]{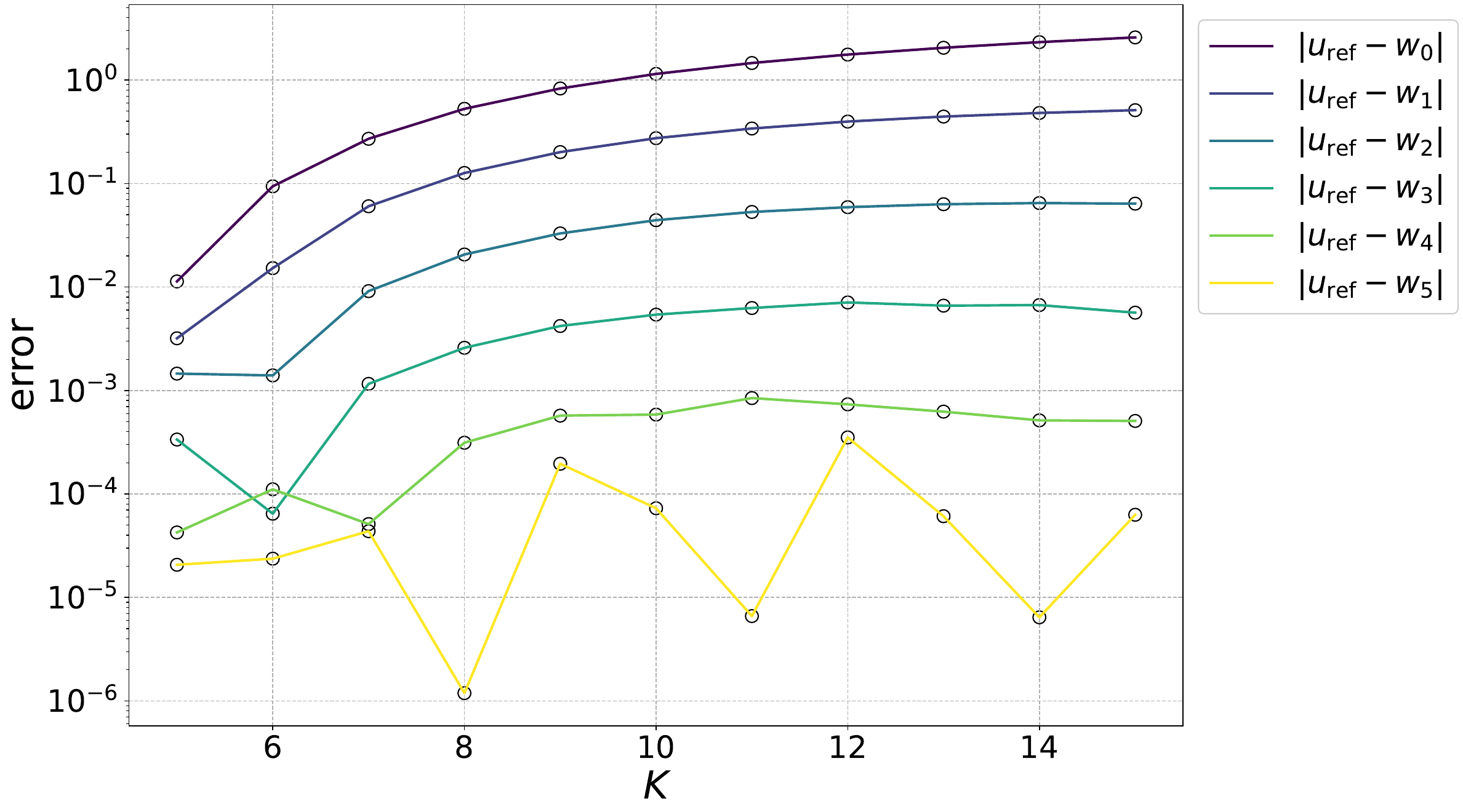}
\caption{The approximation errors $|u_{\rm ref}(0,x)-w_m(0,x)|$ for $m\in \{0,1,\cdots,5\}$ at initial state $x=10.0$ and eleven different strikes $K\in\{5.0,6.0,\cdots,15.0\}$.}
\label{Fig_Merton}
\end{figure}

\subsubsection{Multi-dimensional case}


We next examine the multi-dimension setting ($d\in \{2,3,\cdots\}$).
With the problem parameters $d=10$, $T=1.0$, $r=0.04$, $\sigma=0.25\mathbb{I}_d$, $\lambda=0.5$, $\mu_J=0.5\mathbbm{1}_d$, and $\sigma_J=0.5\mathbb{I}_d$, we approximate $w_m(0,{\bf x})$ for the first five iterations $m\in\{0,1,\cdots,5\}$ by the standard Monte Carlo estimation based on $10^7$ iid paths with the aid of its semi-analytic expression \eqref{Merton_wm} 
and then estimate the approximation errors $|u_{\rm ref}(0,{\bf x})-w_m(0,{\bf x})|$ on the basis of the reference values $u_{\rm ref}(0,{\bf x})$ obtained in advance  due to the formula \eqref{CVA_ref_E} by the standard Monte Carlo estimation with $10^8$ iid paths.
In light of the numerical results presented in Table \ref{Tbl_Merton_d} and Figure \ref{Fig_Merton_d}, it is quite encouraging that iterating only a few steps seem sufficient to achieve a relative error of $1\%$, even with the increased problem dimension, as expected in Theorem \ref{thm1}.

\begin{table}[H]
  \begin{center}\small
    \begin{tabular}{c||cc|cc|cc}
      \hline & \multicolumn{2}{c|}{$K=5.0$} & \multicolumn{2}{c|}{$K=10.0$} & \multicolumn{2}{c}{$K=15.0$}\\
      \hline
      & value & error& value & error& value & error\\\hline\hline
      $u_{\rm ref}(0,{\bf x})$ & $0.000005$ & $-$& $1.746767$ & $-$& $5.060325$ & $-$\\\hline
      $w_0(0,{\bf x})$ & $0.000011$ & $0.000006$ ($120.00$\%) & $3.257313$ & $1.510746$ ($86.49$\%) & $8.257382$ & $3.197053$ ($63.18$\%) \\ 
      $w_1(0,{\bf x})$ & $0.000004$ & $0.000001$ ($20.00$\%) & $1.227868$ & $0.518900$ ($29.71$\%) & $4.301782$ & $0.758543$ ($14.99$\%) \\
      $w_2(0,{\bf x})$ & $0.000006$ & $0.000001$ ($20.00$\%) & $1.859807$ & $0.113039$ ($6.47$\%) & $5.152434$ & $0.092109$ ($1.82$\%) \\
      $w_3(0,{\bf x})$ & $0.000006$ & $0.000000$ ($0.00$\%) & $1.728550$ & $0.018217$ ($1.04$\%) & $5.057154$ & $0.003171$ ($0.06$\%) \\
      $w_4(0,{\bf x})$ & $0.000005$ & $0.000000$ ($0.00$\%) & $1.749069$ & $0.002301$ ($0.13$\%) & $5.059185$ & $0.001140$ ($0.02$\%) \\
      $w_5(0,{\bf x})$ & $0.000005$ & $0.000000$ ($0.00$\%) & $1.746431$ & $0.000337$ ($0.02$\%) & $5.060723$ & $0.000398$ ($0.01$\%) \\\hline
      \end{tabular}
      \caption{The reference values $u_{\rm ref}(0,{\bf x})$, the values from the recursion $\{w_m(0,{\bf x})\}_{m\in \{0,1,\cdots,5\}}$, and the approximation errors $|u_{\rm ref}(0,{\bf x})-w_m(0,{\bf x})|$ for $m\in \{0,1,\cdots,5\}$ at initial state ${\bf x}=10.0\mathbbm{1}_d$ and three different strikes $K\in \{5.0,\,10.0,\,15.0\}$, with relative errors in brackets.}
      \label{Tbl_Merton_d}
    \end{center}
  \end{table} 

\begin{figure}[H]
\centering
\includegraphics[width=9cm]{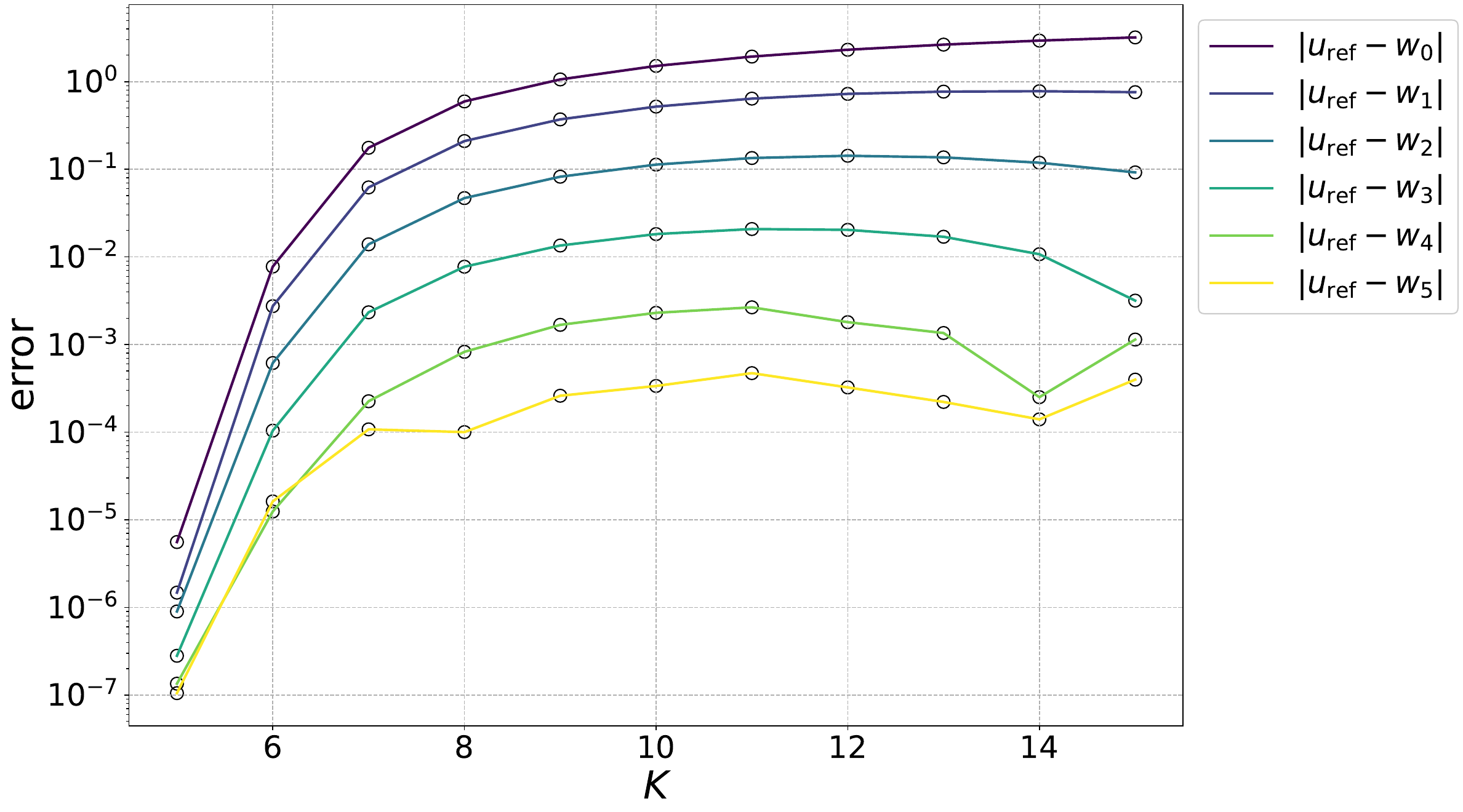}
\caption{The approximation errors $|u_{\rm ref}(0,{\bf x})-w_m(0,{\bf x})|$ for $m\in \{0,1,\cdots,5\}$ at initial state ${\bf x}=10.0\mathbbm{1}_d$ and eleven different strikes $K\in\{5.0,6.0,\cdots,15.0\}$.}
\label{Fig_Merton_d}
\end{figure}



\subsection{Second example}\label{section second example}

In this section, we examine the developed forward scheme and its machine learning-based implementation (Section \ref{section forward scheme}) with an emphasis on the problem dimension and the jump intensity (or, equivalently, the parameter $\eta$ in Theorem \ref{thm1}).
To this end, consider a multi-dimensional version of the FBSDE investigated in \cite[Section 5.2]{GEISS20162123}, with $d_{\bf x}=d_0=d$ for some $d\in \mathbb{N}$ and $d_{\bf y}=d_{\bf z}=1$: for $s\in [0,T]$,
\[
\begin{dcases}
dX_s^{0,{\bf x}} = b ds + \sigma dB_s + \mathbbm{1}_dcdN_s ,\\
dY_s^{0,{\bf x}}=-(\alpha Y_s^{0,{\bf x}}+\beta \langle Z_s^{0,{\bf x}},\mathbbm{1}_d/d\rangle+\rho U_s^{0,{\bf x}})ds+\langle Z_s^{0,{\bf x}}, dB_s\rangle+ U_s^{0,{\bf x}}(dN_s-\lambda ds),\\
X_0^{0,{\bf x}}= {\bf x}\in \mathbb{R}^d,\quad Y_T^{0,{\bf x}}=\Phi(X_T^{0,{\bf x}}),
\end{dcases}
\]
where $b\in \mathbb{R}^d$, $\sigma\in\mathbb{R}^{d\times d}$, $c>0$, $\{N_s:\,s\in [0,T]\}$ denotes the Poisson process with intensity $\lambda>0$, independent of the $d$-dimensional standard Brownian motion $\{B_s:\,s\in [0,T]\}$, and $\Phi({\bf x})=e^{\langle {\bf x},\mathbbm{1}_d/d\rangle}$.
Given the constant intensity function $\lambda(\cdot,\cdot)\equiv \lambda>0$, the parameter $\eta$, denoting its upper bound (Assumption \ref{standing assumption}), can be regarded as identical to the constant intensity, that is, $\eta=\lambda$.  
In accordance with the formulation \eqref{PIDE}, its associated PIDE can be written as: $(s,{\bf x})\in[0,T)\times \mathbb{R}^d,$
\begin{multline*}
\frac{\partial}{\partial s} u(s,{\bf x})+\langle b, \nabla u(s,{\bf x})\rangle +\frac{1}{2}\mathrm{tr}\left[\sigma^{\otimes 2}\mathrm{Hess}(u(s,{\bf x}))\right]\\
+(u(s,{\bf x}+c)-u(s,{\bf x}))\lambda+f(u(s,{\bf x}),(\nabla u\sigma)(s,{\bf x}),(u(s,{\bf x}+c)-u(s,{\bf x}))\lambda)=0, 
\end{multline*}
with $u(T,{\bf x})=\Phi({\bf x})$.
Here, the driver is given by $f(y,{\bf z},\gamma)=\textstyle{\alpha y + \beta \langle {\bf z},\mathbbm{1}_d/d\rangle+\rho \gamma/\lambda}$ for $(y,{\bf z},\gamma)\in\mathbb{R}\times \mathbb{R}^d\times \mathbb{R}$.
This formulation depends on the argument ${\bf z}$ and $\gamma$, making this problem setting somewhat more general than that examined in Section \ref{subsection first example}.
We note that the reference value $u(0,{\bf x})$ can be found easily with the aid of the analytic solution, which is available when $d=1$ \cite[Section 5.2]{GEISS20162123} (as employed shortly in Section \ref{ex1 toy 00}), and even when $d\in \{2,3,\cdots\}$ under suitable restrictions on the problem parameters (Section \ref{ex1 high-dimensional}).
In accordance with the formulations \eqref{PIDE_0} and \eqref{PIDE_m}, the recursion $\{w_m\}_{m\in\mathbb{N}_0}$ can be written as: for $(s,{\bf x})\in[0,T)\times \mathbb{R}^d$, 
\[
\frac{\partial}{\partial s} w_0(s,{\bf x})+\langle b, \nabla w_0(s,{\bf x})\rangle+\frac{1}{2}\mathrm{tr}\left[\sigma^{\otimes 2}\mathrm{Hess}(w_0(s,{\bf x}))\right]=0,
\]
and
\begin{multline*}
\frac{\partial}{\partial s} w_m(s,{\bf x})+\langle b, \nabla w_m(s,{\bf x})\rangle+\frac{1}{2}\mathrm{tr}\left[\sigma^{\otimes 2}\mathrm{Hess}(w_m(s,{\bf x}))\right]+(w_{m-1}(s,{\bf x}+c)-w_{m-1}(s,{\bf x}))\lambda\\
+f(w_{m-1}(s,{\bf x}),(\nabla w_{m-1}\sigma)(s,{\bf x}),(w_{m-1}(s,{\bf x}+c)-w_{m-1}(s,{\bf x})\lambda))=0,
\end{multline*}
with $w_m(T,{\bf x})=\Phi({\bf x})$ for all $m\in\mathbb{N}_0$.
As evident, the first one can be solved explicitly for $w_0$, while not for $w_m$ even with $\beta=\rho=0$, to the best of our knowledge.

\subsubsection{One-dimensional case}
\label{ex1 toy 00}

For better illustration, we start with the one-dimensional setting ($d=1$), with the problem parameters fixed as $T=2.0$, $b=-0.1$, $\sigma=0.1$, $c=0.2$, $\lambda=3.0$, $\alpha=0.3$, $\beta=0.3$, and $\rho=0.2$, in consistency with \cite{GEISS20162123}.
To approximate the recursion $\{w_m\}_{m\in\mathbb{N}_0}$, we employ the least squares Monte Carlo method (Section \ref{subsection LSMC}) with the number of paths $M=10^7$ and the number of time discretization $n=2^6$ throughout, that is, even for $w_0$ (despite the availability of its analytic solution) to maintain consistency with the approximation of all subsequent iterations $\{w_m\}_{m\in\mathbb{N}}$. 
We present in Table \ref{Tbl_eg2_1d_lsmc_GL} the approximation errors $|u_{\rm ref}(0,x)-w_m(0,x)|$ for $m\in \{0,1,\cdots,8\}$, where the reference values $u_{\rm ref}(0,x)$ are due to its analytical solution, and each step of the recursion $\{w_m(0,x)\}_{m\in\mathbb{N}_0}$ is approximated as average of five iid runs of the least squares Monte Carlo method (Section \ref{subsection LSMC}) with the basis functions $\{v_j^{\mathrm{basis}}(x)\}_{j\in \{1,\cdots,4\}}=\{1,x,x^2,\Phi(x)\}$, along with plots in Figure \ref{Fig_eg2_1d_lsmc_GL} at an extended set of initial states.
We note that each of the five runs for computing the eighth step $w_8(0,x)$ takes roughly $120$ seconds.
As expected with reference to Theorem \ref{thm1}, it appears that the proposed forward scheme converges to the solution $u$ even when the driver $f$ depends on the arguments $z$ and $\gamma$.
Nonetheless, the proposed forward scheme has required not only a few but seven iterations to achieve a relative error of $1\%$ at those initial states, primarily due to the high jump intensity $\lambda=3.0$, which can significantly influence the trajectory.

\begin{table}[H]
\begin{center}\small
\begin{tabular}{c||cc|cc|cc}
\hline  & \multicolumn{2}{c|}{$x=-0.5$} & \multicolumn{2}{c|}{$x=0.0$} & \multicolumn{2}{c}{$x=0.5$}\\
\hline
& value & error& value & error& value & error\\\hline\hline
$u_{\rm ref}(0,x)$ & $4.002734$ & $-$& $6.599393$ & $-$& $10.880559$ & $-$\\\hline
$w_0(0,x)$ & $0.501579$&$3.501155$ ($87.47$\%) &$0.826964$&$5.772429$ ($87.47$\%) &$1.363433$&$9.517127$ ($87.47$\%)\\
$w_1(0,x)$ & $1.543723$&$2.459011$ ($61.43$\%)&$2.545168$&$4.054224$ ($61.43$\%)&$4.196274$&$6.684285$ ($61.43$\%)\\
$w_2(0,x)$ & $2.600356$&$1.402378$ ($35.04$\%)&$4.287259$&$2.312134$ ($35.04$\%)&$7.068500$&$3.812059$ ($35.04$\%)\\
$w_3(0,x)$ & $3.336850$&$0.665884$ ($16.64$\%)&$5.501532$&$1.097860$ ($16.64$\%)&$9.070503$&$1.810056$ ($16.64$\%)\\
$w_4(0,x)$ & $3.728353$&$0.274381$ ($6.85$\%)&$6.147005$&$0.452388$ ($6.85$\%)&$10.134711$&$0.745848$ ($6.85$\%)\\
$w_5(0,x)$ & $3.896475$&$0.106259$ ($2.65$\%)&$6.424196$&$0.175196$ ($2.65$\%)&$10.591722$&$0.288838$ ($2.65$\%)\\
$w_6(0,x)$ & $3.956764$&$0.045970$ ($1.15$\%)&$6.523593$&$0.075800$ ($1.15$\%)&$10.755599$&$0.124960$ ($1.15$\%)\\
$w_7(0,x)$ & $3.975274$&$0.027460$ ($0.69$\%)&$6.554104$&$0.045289$ ($0.69$\%)&$10.805907$&$0.074652$ ($0.69$\%)\\
$w_8(0,x)$ & $3.980245$&$0.022489$ ($0.56$\%)&$6.562304$&$0.037089$ ($0.56$\%)&$10.819420$&$0.061139$ ($0.56$\%)\\
\hline
\end{tabular}
\caption{The reference values $u_{\rm ref}(0,x)$, the values from the recursion $\{w_m(0,x)\}_{m\in \{0,1,\cdots,8\}}$, and the approximation errors $|u_{\rm ref}(0,x)-w_m(0,x)|$ for $m\in \{0,1,\cdots,8\}$ with a high jump intensity of $\lambda=3.0$ at three different initial states $x\in\{-0.5,0.0,0.5\}$ by the least squares Monte Carlo method (Section \ref{subsection LSMC}), with relative errors in brackets.}
\label{Tbl_eg2_1d_lsmc_GL}
\end{center}
\end{table}

\begin{figure}[H]
\centering
\includegraphics[width=9cm]{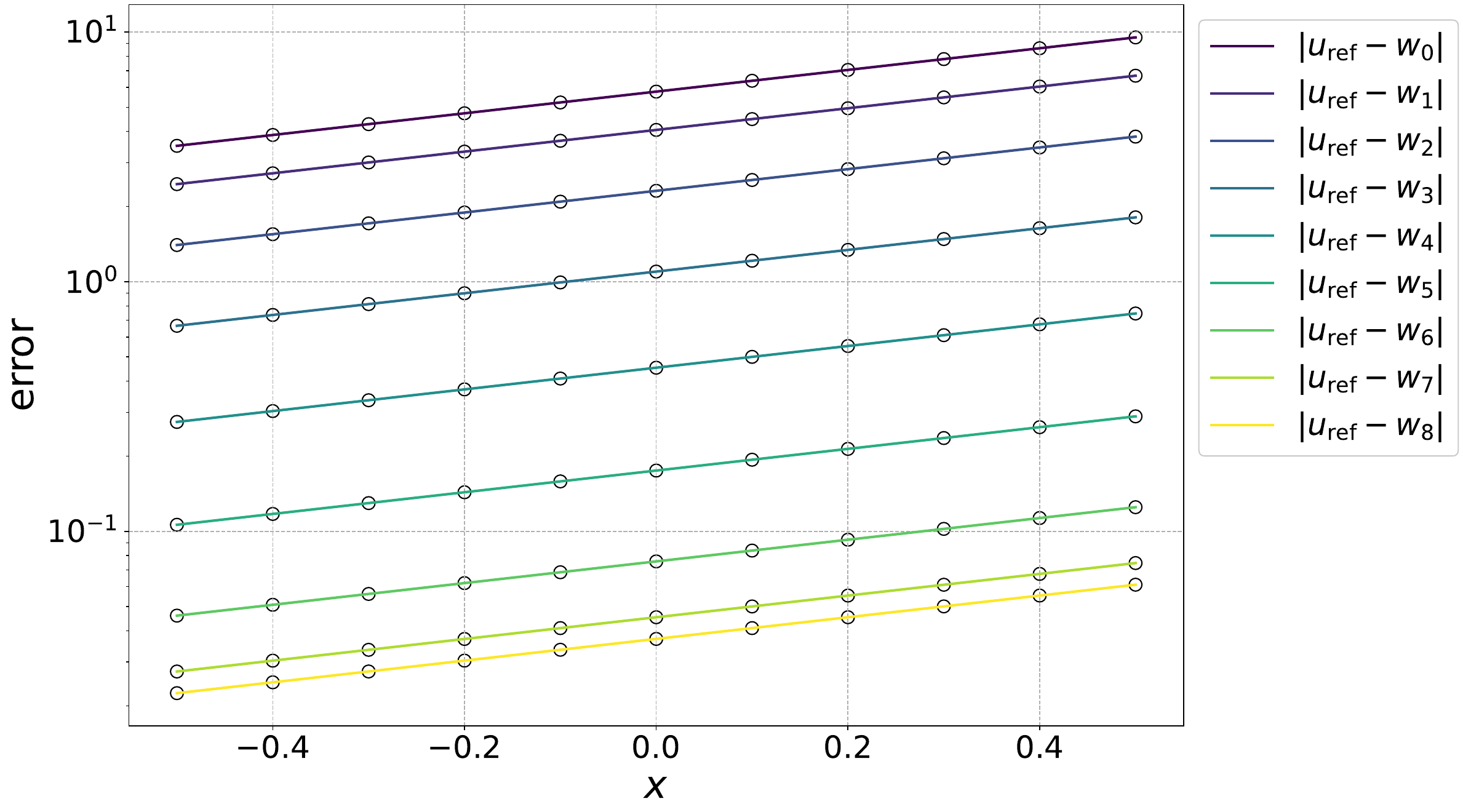}
\caption{The approximation errors $|u_{\rm ref}(0,x)-w_m(0,x)|$ for $m\in \{0,1,\cdots,8\}$ with a high jump intensity of $\lambda=3.0$ at eleven different initial states $x\in \{-0.5,-0.4,\cdots,0.5\}$ by the least squares Monte Carlo method (Section \ref{subsection LSMC}).}
\label{Fig_eg2_1d_lsmc_GL}
\end{figure}

Next, to demonstrate the effect of the jump intensity on the performance of the proposed forward scheme, we examine lower jump intensities while keeping all other problem parameters unchanged for a fair comparison.
In Tables \ref{Tbl_eg2_1d_lsmc_lambda_1.0} and \ref{Tbl_eg2_1d_lsmc_lambda_0.5}, we present numerical results for the jump intensities $\lambda=1.0$ and $\lambda=0.5$, respectively.
For further illustration, we also provide plots of the approximation errors in Figure \ref{Fig_eg2_1d_lsmc_lambda}, comparing those two cases side by side.
Clearly, only with the value of the jump intensity $\lambda$ amended, the computing times required here remain similar to the case of $\lambda=3.0$, that is, each of the five runs for computing the eighth step $w_8(0,x)$ takes roughly $120$ seconds.
Unlike the high intensity ($\lambda=3.0$ corresponding to Table \ref{Tbl_eg2_1d_lsmc_GL} and Figure \ref{Fig_eg2_1d_lsmc_GL}), those lower intensities do not ``shape'' the trajectory but rather only introduce occasional irregularities into the diffusive dynamics, such as rare credit events occurring at most once year ($\lambda=1.0$) or even less frequently ($\lambda=0.5$) on average.
Now, with those lower jump intensities, the proposed forward scheme has required only five and four iterations, respectively, to achieve a relative error of $1\%$ across all presented initial states, as depicted in Figure \ref{Fig_eg2_1d_lsmc_lambda}.
Those results contrast sharply with the roughly seven iterations required for the high jump intensity $\lambda=3.0$ (Figure \ref{Fig_eg2_1d_lsmc_GL}).
In summary, this is a consistent consequence of the vanishing upper bound for the approximation error with respect to lower jump intensities, as expected in Theorem \ref{thm1}.

\begin{table}[H]
\begin{center}\small
\begin{tabular}{c||cc|cc|cc}
\hline  & \multicolumn{2}{c|}{$x=-0.5$} & \multicolumn{2}{c|}{$x=0.0$} & \multicolumn{2}{c}{$x=0.5$}\\
\hline
& value & error& value & error& value & error\\\hline\hline
$u_{\rm ref}(0,x)$ & $1.650976$ & $-$& $2.721999$ & $-$& $4.487818$ & $-$\\\hline
$w_0(0,x)$ & $0.501579$ & $1.149397$ ($69.62\%$) & $0.826964$ & $1.895035$ ($69.62\%$) & $1.363433$ & $3.124385$ ($69.62\%$) \\
$w_1(0,x)$ & $1.099348$ & $0.551628$ ($33.41\%$) & $1.812519$ & $0.909480$ ($33.41\%$) & $2.988339$ & $1.499479$ ($33.41\%$) \\
$w_2(0,x)$ & $1.447513$ & $0.203463$ ($12.32\%$) & $2.386544$ & $0.335455$ ($12.32\%$) & $3.934748$ & $0.553070$ ($12.32\%$) \\
$w_3(0,x)$ & $1.587145$ & $0.063831$ ($3.87\%$) & $2.616756$ & $0.105243$ ($3.87\%$) & $4.314304$ & $0.173514$ ($3.87\%$) \\
$w_4(0,x)$ & $1.629841$ & $0.021135$ ($1.28\%$) & $2.687151$ & $0.034848$ ($1.28\%$) & $4.430365$ & $0.057453$ ($1.28\%$) \\
$w_5(0,x)$ & $1.640434$ & $0.010542$ ($0.64\%$) & $2.704616$ & $0.017383$ ($0.64\%$) & $4.459161$ & $0.028657$ ($0.64\%$) \\
$w_6(0,x)$ & $1.642646$ & $0.008330$ ($0.50\%$) & $2.708263$ & $0.013736$ ($0.50\%$) & $4.465173$ & $0.022645$ ($0.50\%$) \\
$w_7(0,x)$ & $1.643045$ & $0.007931$ ($0.48\%$) & $2.708920$ & $0.013079$ ($0.48\%$) & $4.466256$ & $0.021562$ ($0.48\%$) \\
$w_8(0,x)$ & $1.643107$ & $0.007869$ ($0.48\%$) & $2.709024$ & $0.012975$ ($0.48\%$) & $4.466427$ & $0.021391$ ($0.48\%$) \\
\hline
\end{tabular}
\caption{The reference values $u_{\rm ref}(0,x)$, the values from the recursion $\{w_m(0,x)\}_{m\in \{0,1,\cdots,8\}}$, and the approximation errors $|u_{\rm ref}(0,x)-w_m(0,x)|$ for $m\in \{0,1,\cdots,8\}$ with a lower jump intensity of $\lambda=1.0$ at three different initial states $x\in\{-0.5,0.0,0.5\}$ by the least squares Monte Carlo method (Section \ref{subsection LSMC}), with relative errors in brackets.}
\label{Tbl_eg2_1d_lsmc_lambda_1.0}
\end{center}
\end{table}

\begin{table}[H]
\begin{center}\small
\begin{tabular}{c||cc|cc|cc}
\hline  & \multicolumn{2}{c|}{$x=-0.5$} & \multicolumn{2}{c|}{$x=0.0$} & \multicolumn{2}{c}{$x=0.5$}\\
\hline
& value & error& value & error& value & error\\\hline\hline
$u_{\rm ref}(0,x)$ & $1.323082$ & $-$& $2.181393$ & $-$& $3.596510$ & $-$\\\hline
$w_0(0,x)$ & $0.501579$ & $0.821503$ ($62.09\%$) & $0.826964$ & $1.354430$ ($62.09\%$) & $1.363433$ & $2.233077$ ($62.09\%$) \\
$w_1(0,x)$ & $0.988255$ & $0.334827$ ($25.31\%$) & $1.629356$ & $0.552037$ ($25.31\%$) & $2.686355$ & $0.910155$ ($25.31\%$) \\
$w_2(0,x)$ & $1.219152$ & $0.103930$ ($7.86\%$) & $2.010040$ & $0.171354$ ($7.86\%$) & $3.313996$ & $0.282514$ ($7.86\%$) \\
$w_3(0,x)$ & $1.294617$ & $0.028465$ ($2.15\%$) & $2.134462$ & $0.046932$ ($2.15\%$) & $3.519133$ & $0.077376$ ($2.15\%$) \\
$w_4(0,x)$ & $1.313421$ & $0.009661$ ($0.73\%$) & $2.165462$ & $0.015931$ ($0.73\%$) & $3.570246$ & $0.026264$ ($0.73\%$) \\
$w_5(0,x)$ & $1.317225$ & $0.005857$ ($0.44\%$) & $2.171737$ & $0.009657$ ($0.44\%$) & $3.580590$ & $0.015920$ ($0.44\%$) \\
$w_6(0,x)$ & $1.317876$ & $0.005206$ ($0.39\%$) & $2.172809$ & $0.008584$ ($0.39\%$) & $3.582358$ & $0.014151$ ($0.39\%$) \\
$w_7(0,x)$ & $1.317973$ & $0.005109$ ($0.39\%$) & $2.172968$ & $0.008425$ ($0.39\%$) & $3.582620$ & $0.013890$ ($0.39\%$) \\
$w_8(0,x)$ & $1.317985$ & $0.005097$ ($0.39\%$) & $2.172989$ & $0.008404$ ($0.39\%$) & $3.582654$ & $0.013856$ ($0.39\%$) \\
\hline
\end{tabular}
\caption{The reference values $u_{\rm ref}(0,x)$, the values from the recursion $\{w_m(0,x)\}_{m\in \{0,1,\cdots,8\}}$, and the approximation errors $|u_{\rm ref}(0,x)-w_m(0,x)|$ for $m\in \{0,1,\cdots,8\}$ with an even lower jump intensity of $\lambda=0.5$ at three different initial states $x\in\{-0.5,0.0,0.5\}$ by the least squares Monte Carlo method (Section \ref{subsection LSMC}), with relative errors in brackets.}
\label{Tbl_eg2_1d_lsmc_lambda_0.5}
\end{center}
\end{table}
\begin{figure}[H]
\centering
\begin{tabular}{cc}
\includegraphics[height=5cm]{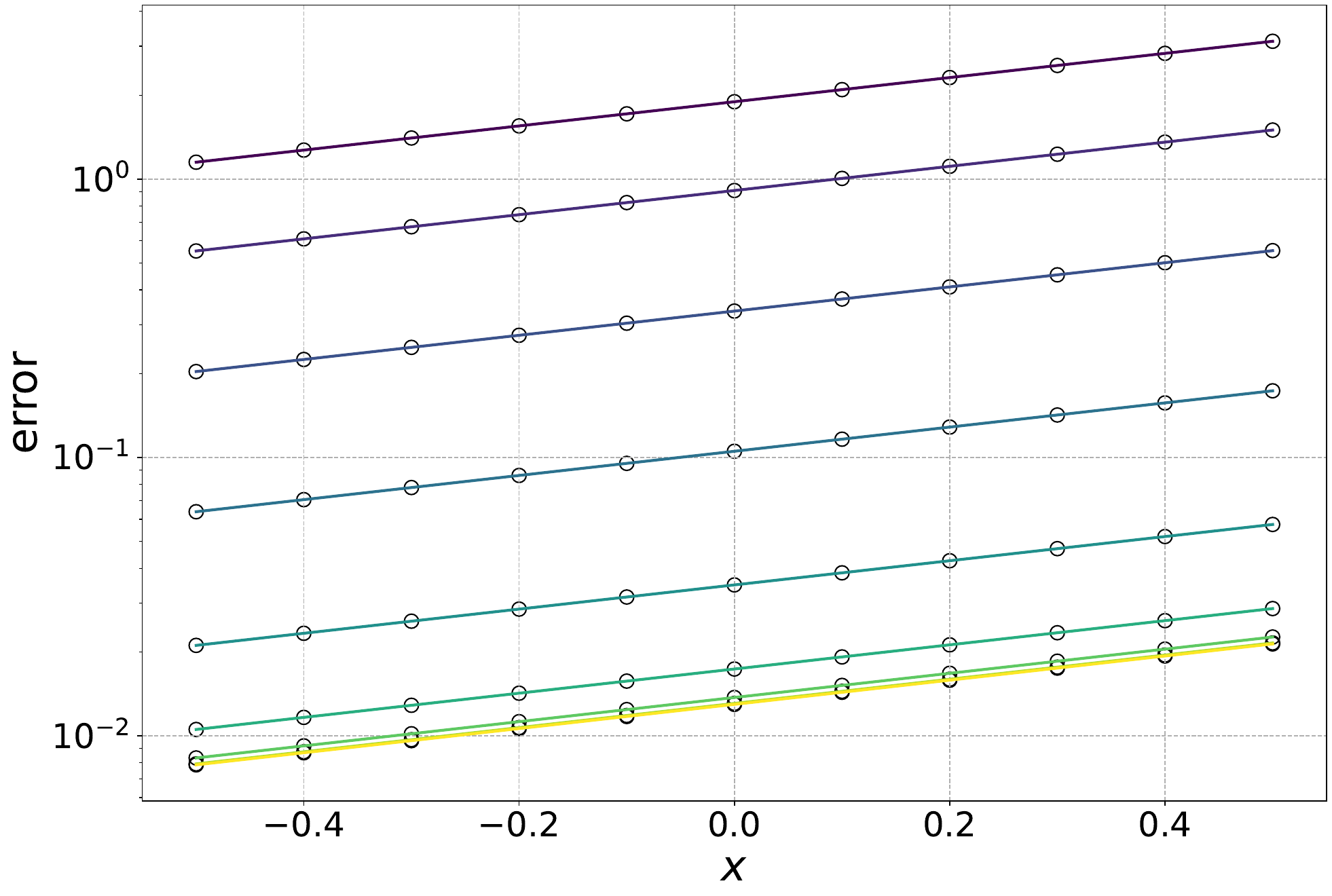} &
\includegraphics[height=5cm]{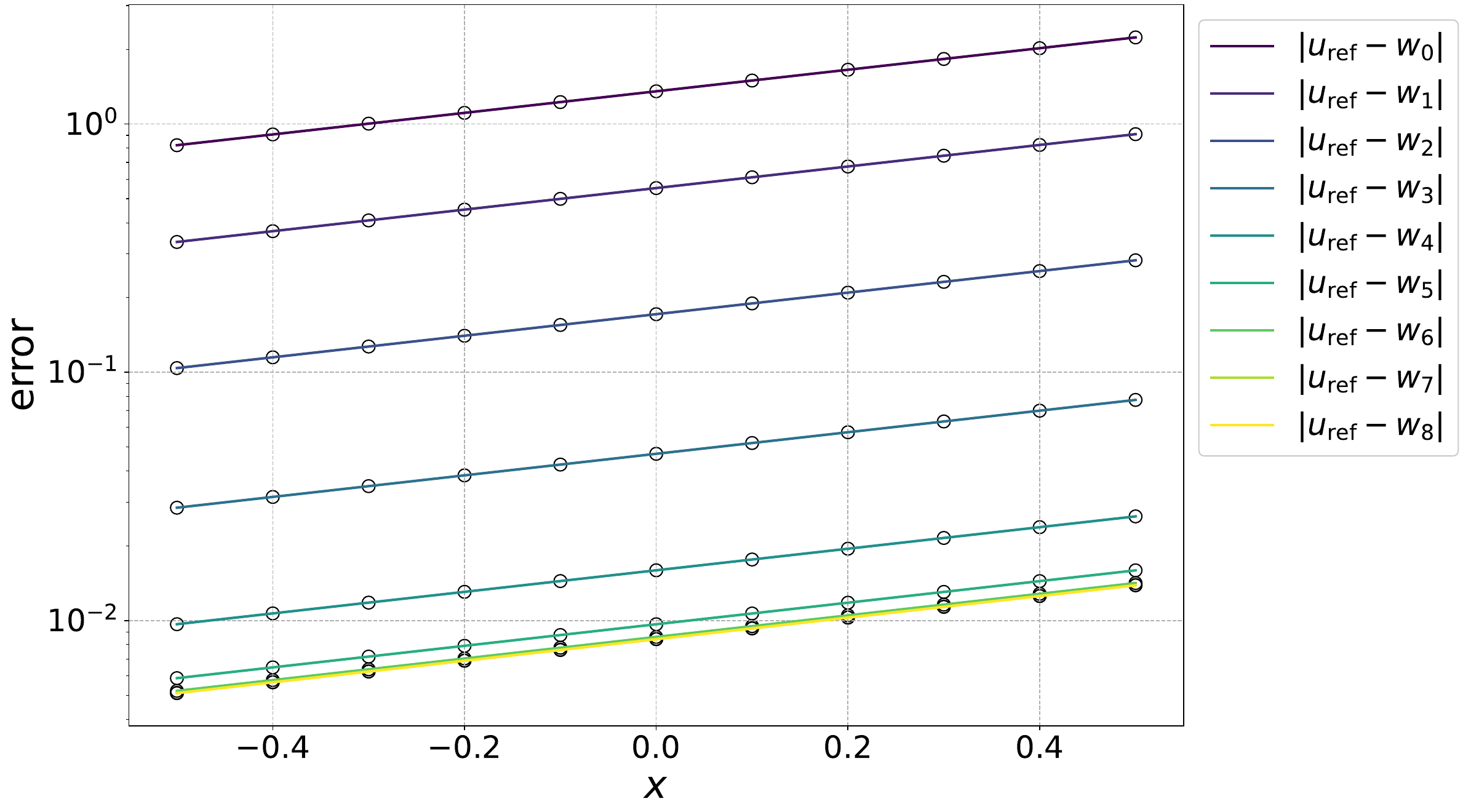} \\
(a) $\lambda=1.0$ & (b) $\lambda=0.5$
\end{tabular}
\caption{The approximation errors $|u_{\rm ref}(0,x)-w_m(0,x)|$ for $m\in \{0,1,\cdots,8\}$ with lower jump intensities of (a) $\lambda=1.0$ and (b) $\lambda=0.5$ at eleven different initial states $x\in \{-0.5,-0.4,\cdots,0.5\}$ by the least squares Monte Carlo method (Section \ref{subsection LSMC}).}
\label{Fig_eg2_1d_lsmc_lambda}
\end{figure}

\subsubsection{Multi-dimensional case}
\label{ex1 high-dimensional}

Finally, we examine the multi-dimensional setting ($d\in \{2,3,\cdots\}$).
We have fixed the drift coefficient $b$ to have all identical components (that is, $b=b_0 \mathbbm{1}_d$ with $b_0\in\mathbb{R}$), the diffusion matrix $\sigma$ to be diagonal with all identical components (that is, $\sigma=\sigma_0\mathbb{I}_d$ with $\sigma_0>0$), and $\beta=\rho=0$, so that the solution $u(0,{\bf x})$ at ${\bf x}=(x_{\mathrm{ini}},\cdots,x_{\mathrm{ini}})^{\top}$ can be expressed in closed form for presenting its reference values without numerical issues, available as follows:
\begin{equation}\label{u ref 5.1.3}
 u(0,{\bf x})=\exp\left[x_{\mathrm{ini}}+b_0T+\left(\alpha+\frac{\sigma_0^2}{2d}\right)T+(e^c-1)\lambda T\right].
\end{equation}
If one-dimensional ($d=1$), then semi-analytic expressions to the recursion $\{w_m\}_{m\in\mathbb{N}_0}$ are available as
\[
w_0(0,x)=e^{x+(b+\sigma^2/2)T},\quad 
w_m(0,x)=\sum_{k=0}^mw_0(0,x+kc) \frac{(\lambda T)^k}{k!}\sum_{n=0}^{m-k}(-\alpha+\lambda)^n \frac{(-T)^n}{n !},
\]
for all $m\in\mathbb{N}$.
To the best of our knowledge, however, this is not the case for the multi-dimensional setting.

First, in Table \ref{Tbl_eg2_10d_lsmc} and Figure \ref{Fig_eg2_10d_lsmc}, we present numerical results for the $10$-dimensional setting ($d=10$), with problem parameters $T=1.0$, $b_0=0.1$, $\sigma_0=0.2$, $c=0.2$, $\lambda=0.5$, $\alpha=0.1$, $\beta=0$, and $\rho=0$, by the least squares Monte Carlo method (Section \ref{subsection LSMC}) with the basis functions $\{v_j^{\mathrm{basis}}({\bf x})\}_{j\in \{1,\cdots,67\}}=\{1,x_1,\cdots,x_{10},x_1^2,x_1x_2,\cdots,x_{9}x_{10},x_{10}^2,\Phi({\bf x})\}$, where the presented values from the recursion $\{w_m\}_{m\in\mathbb{N}_0}$ are the averages of five iid runs of the least squares Monte Carlo method with the number of paths $M=10^7$ and the number of time discretization $n=2^4$.
We note that each of the five runs for computing the third step $w_3(0,{\bf x})$ takes roughly $170$ seconds.
Let us remark that we employ only the least squares Monte Carlo method here, and not the neural network-based scheme (Section \ref{subsection neural network}) as of yet, because, as pointed out towards the end of Section \ref{subsection neural network}, the problem dimension $d=10$ does not sufficiently justify the substantial time required for the initial setup of the neural networks. 
In the present problem, a dimension of roughly 15 or higher would warrant the application of the neural network-based scheme.

\begin{table}[H]
\begin{center}\small
\begin{tabular}{c||cc|cc|cc}
\hline & \multicolumn{2}{c|}{$x_{\mathrm{ini}}=-0.5$} & \multicolumn{2}{c|}{$x_{\mathrm{ini}}=0.0$} & \multicolumn{2}{c}{$x_{\mathrm{ini}}=0.5$}\\
\hline
& value & error& value & error& value & error\\\hline\hline
$u_{\rm ref}(0,{\bf x})$ & $0.758919$ & $-$ & $1.251247$ & $-$ & $2.062957$ & $-$ \\ \hline
$w_0(0,{\bf x})$ & $0.671660$ & $0.087259$ ($11.50\%$) & $1.107380$ & $0.143866$ ($11.50\%$) & $1.825761$ & $0.237195$ ($11.50\%$) \\
$w_1(0,{\bf x})$ & $0.753616$ & $0.005303$ ($0.70\%$) & $1.242501$ & $0.008745$ ($0.70\%$) & $2.048540$ & $0.014417$ ($0.70\%$) \\
$w_2(0,{\bf x})$ & $0.758353$ & $0.000567$ ($0.07\%$) & $1.250310$ & $0.000936$ ($0.07\%$) & $2.061413$ & $0.001544$ ($0.07\%$) \\
$w_3(0,{\bf x})$ & $0.758553$ & $0.000366$ ($0.05\%$) & $1.250644$ & $0.000603$ ($0.05\%$) & $2.061961$ & $0.000996$ ($0.05\%$) \\
\hline
\end{tabular}
\caption{The reference values $u_{\rm ref}(0,{\bf x})$, the values from the recursion $\{w_m(0,{\bf x})\}_{m\in\mathbb{N}_0}$, and the approximation errors $|u_{\rm ref}(0,{\bf x})-w_m(0,{\bf x})|$ for $m\in \{0,1,\cdots,3\}$ at three different initial values ${\bf x}=(x_{\rm ini},\cdots,x_{\rm ini})^{\top}$ with $x_{\rm ini}\in\{-0.5,0.0,0.5\}$ by the least squares Monte Carlo method (Section \ref{subsection LSMC}), with relative errors in brackets.}
\label{Tbl_eg2_10d_lsmc}
\end{center}
\end{table}
\begin{figure}[H]
\centering
\includegraphics[width=9cm]{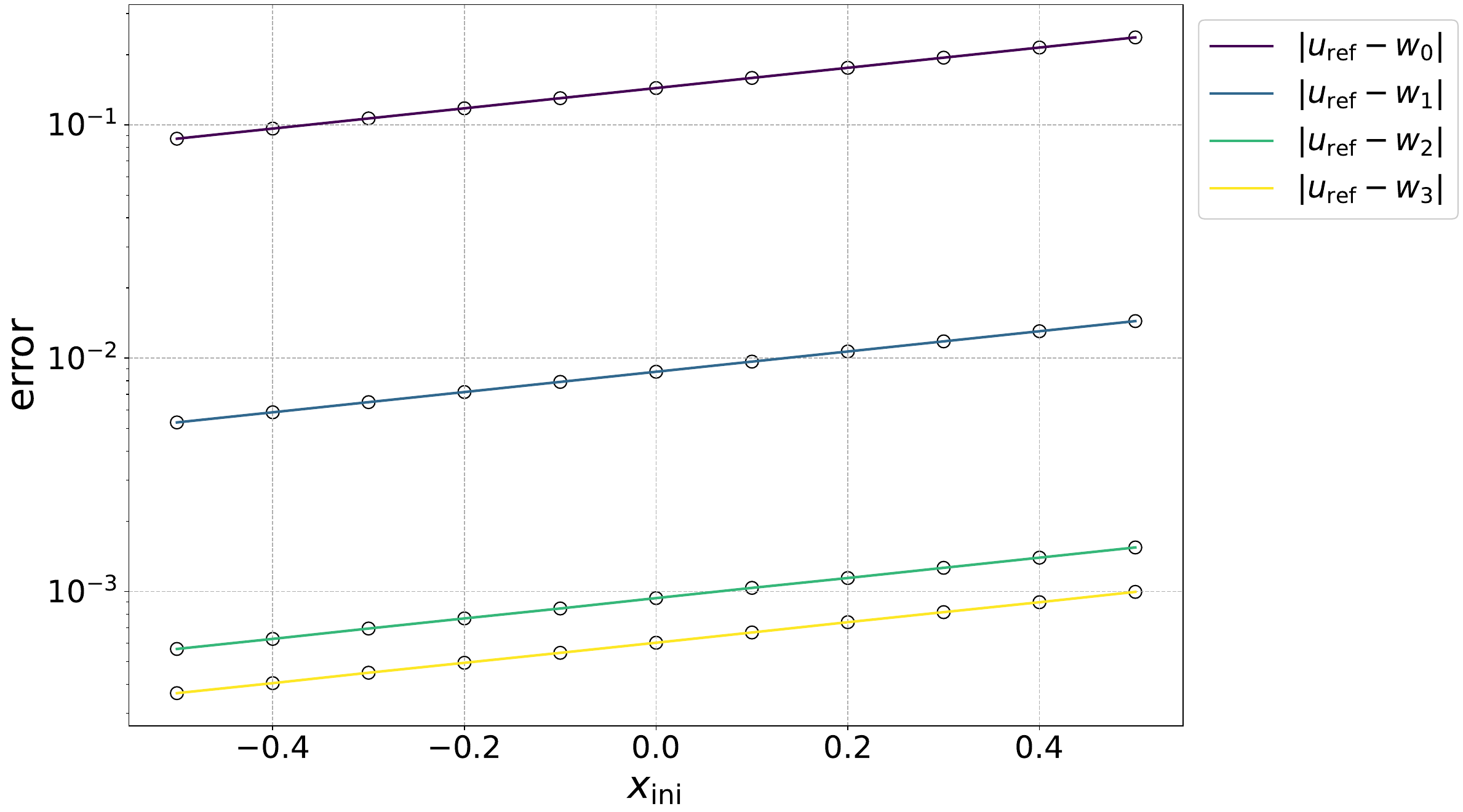}
\caption{The approximation errors $|u_{\rm ref}(0,{\bf x})-w_m(0,{\bf x})|$ for $m\in \{0,1,2,3\}$ at eleven different initial states ${\bf x}=(x_{\rm ini},\cdots,x_{\rm ini})^{\top}$ with $x\in \{-0.5,-0.4,\cdots,0.5\}$ by the least squares Monte Carlo method (Section \ref{subsection LSMC}).}
\label{Fig_eg2_10d_lsmc}
\end{figure}

We next increase the problem dimension to $d=100$ with problem parameters $T=1.0$, $b_0=0.1$, $\sigma_0=0.1$, $c=0.2$, $\lambda=0.5$, $\alpha=0.1$, $\beta=0$, and $\rho=0$, where the formula \eqref{u ref 5.1.3} remains valid.
Here, we employ the neural network-based scheme (Section \ref{subsection neural network}) alone, because at this scale, the least squares Monte Carlo method (Section \ref{subsection LSMC}) is effectively infeasible, due to the extremely large number of basis functions required.
Let us stress again that in this context, the neural network-based scheme outperforms the least squares Monte Carlo method, not in terms of computing time, but rather fundamentally in terms of feasibility.

We implement the neural network-based scheme, consisting of 1 input layer ($d$-neuron), two hidden layers ($d+10$ neurons each) and 1 output layer ($1$ neuron) with the hyperparameters fixed as  
the batch size $M=32768$ and the train steps $J=4000$ (and, moreover, with the learning rate $\iota(j)=10^{-2}\mathbbm{1}_{[0,0.3J]}(j)+10^{-3}\mathbbm{1}_{(0.3J,0.6J]}(j)+10^{-4}\mathbbm{1}_{(0.6J,J]}(j)$ for $j\in \{0,1,\cdots,J\}$, which we specify in Algorithm \ref{Algo_NN} in the supplementary materials).
As done in \cite{doi:10.1137/19M1297919}, we have applied the Adam optimizer and the batch normalization technique to enhance computational efficiency.
We present numerical results in Table \ref{Tbl_eg2_100d_NN} and Figure \ref{Fig_eg2_100d_deep}, where the presented values from the recursion $\{w_m\}_{m\in\mathbb{N}_0}$ are the averages of five iid runs by the neural network-based scheme with the number of time discretization $n=2^3$, along with the reference values again based on the formula \eqref{u ref 5.1.3}.
We note that each of the five runs for computing the third step $w_3(0,{\bf x})$ takes roughly $1500$ seconds.
It is promising that even for such large dimensions, the proposed forward scheme has achieved a relative error of $1\%$ across all those initial states after only one iteration.

\begin{table}[H]
\begin{center}\small
\begin{tabular}{c||cc|cc|cc}
\hline & \multicolumn{2}{c|}{$x_{\mathrm{ini}}=-0.5$} & \multicolumn{2}{c|}{$x_{\mathrm{ini}}=0.0$} & \multicolumn{2}{c}{$x_{\mathrm{ini}}=0.5$}\\
\hline
& value & error& value & error& value & error\\\hline\hline
$u_{\rm ref}(0,{\bf x})$ & $0.748688$ & $-$ & $1.234378$ & $-$ & $2.035145$ & $-$ \\ \hline
$w_0(0,{\bf x})$ & $0.670350$ & $0.078338$ ($10.46\%$) & $1.105223$ & $0.129155$ ($10.46\%$) & $1.822209$ & $0.212936$ ($10.46\%$) \\
$w_1(0,{\bf x})$ & $0.744059$ & $0.004629$ ($0.62\%$)  & $1.226736$ & $0.007642$ ($0.62\%$) & $2.022536$ & $0.012609$ ($0.62\%$) \\
$w_2(0,{\bf x})$ & $0.747559$ & $0.001129$ ($0.15\%$)  & $1.232539$ & $0.001839$ ($0.15\%$) & $2.032089$ & $0.003056$ ($0.15\%$) \\
$w_3(0,{\bf x})$ & $0.747672$ & $0.001016$ ($0.14\%$)  & $1.232677$ & $0.001701$ ($0.14\%$) & $2.032365$ & $0.002780$ ($0.14\%$) \\
\hline
\end{tabular}
\caption{The reference values $u_{\rm ref}(0,{\bf x})$, the values from the recursion $\{w_m(0,{\bf x})\}_{m\in\mathbb{N}_0}$, and the approximation errors $|u_{\rm ref}(0,{\bf x})-w_m(0,{\bf x})|$ for $m\in \{0,1,\cdots,3\}$ at three different initial states ${\bf x}=(x_{\rm ini},\cdots,x_{\rm ini})^{\top}$ with $x_{\rm ini}\in\{-0.5,0.0,0.5\}$ by the neural network-based scheme (Section \ref{subsection neural network}), with relative errors in brackets.}
\label{Tbl_eg2_100d_NN}
\end{center}
\end{table}

\begin{figure}[H]
\centering
\includegraphics[width=9cm]{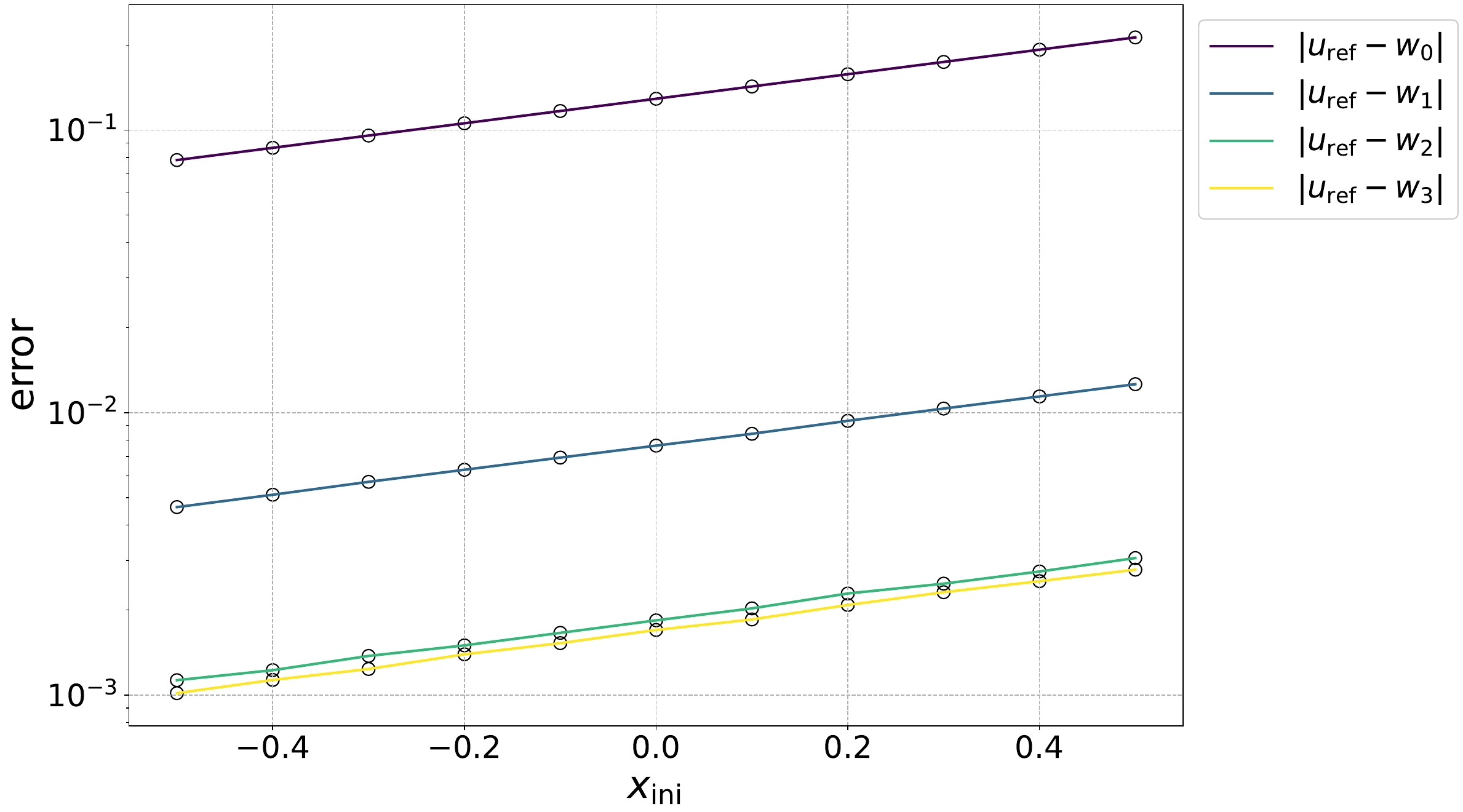}
\caption{The approximation errors $|u_{\rm ref}(0,{\bf x})-w_m(0,{\bf x})|$ for $m\in \{0,1,2,3\}$ at eleven different initial states ${\bf x}=(x_{\rm ini},\cdots,x_{\rm ini})^{\top}$ with $x\in \{-0.5,-0.4,\cdots,0.5\}$ by the neural network-based scheme (Section \ref{subsection neural network}).}
\label{Fig_eg2_100d_deep}
\end{figure}

\section{Proofs}\label{section proofs}

Here, we collect proofs of Theorems \ref{theorem Ytilde}, \ref{theorem section 3.2}, and \ref{thm1}. 
We keep the section fairly self-contained but skip nonessential details in some instances to avoid overloading the section.

\begin{proof}[Proof of Theorem \ref{theorem Ytilde}]
Applying the It\^o formula to $u(s,\widetilde{X}_s^{t,{\bf x}})$, integrating it on $[t,T]$ and then taking the expectation yields
\begin{multline}
u(t,{\bf x})=\mathbb{E}\left[\Phi(\widetilde{X}_T^{t,{\bf x}})+\int_t^Tf(s,\widetilde{X}_s^{t,{\bf x}},u(s,\widetilde{X}_s^{t,{\bf x}}),(\nabla u \sigma)(s,\widetilde{X}_s^{t,{\bf x}}),\gamma(s,\widetilde{X}_s^{t,{\bf x}};u))ds \right.\\
\left. -\int_t^T\int_{\mathbb{R}_0^{d_{\bf z}}}(u(s,\widetilde{X}_s^{t,{\bf x}}+h(s,\widetilde{X}_s^{t,{\bf x}},{\bf z}))-u(s,\widetilde{X}_s^{t,{\bf x}}))\lambda(s,\widetilde{X}_s^{t,{\bf x}})\nu(d{\bf z})ds\right].\label{u_E}
\end{multline}
Hence, the FBSDE without jumps \eqref{FBSDE_without_jump} is related to the PIDE \eqref{PIDE}, as follows:
\[
\widetilde{Y}_s^{t,{\bf x}}=u(s,\widetilde{X}_s^{t,{\bf x}}),\ \ \widetilde{Z}_s^{t,{\bf x}}=(\nabla u \sigma)(s,\widetilde{X}_s^{t,{\bf x}}),\ \ \widetilde{U}_s^{t,{\bf x}}({\bf z})=u(s,\widetilde{X}_s^{t,{\bf x}}+h(s,\widetilde{X}_s^{t,{\bf x}},{\bf z}))-u(s,\widetilde{X}_s^{t,{\bf x}}),
\]
which yields the desired result.
\end{proof}

\begin{proof}[Proof of Theorem \ref{theorem section 3.2}]
Define a sequence of the triplets $\{(\widetilde{Y}^{m,t,{\bf x}},\widetilde{Z}^{m,t,{\bf x}},\widetilde{U}^{m,t,{\bf x}})\}_{m\in\mathbb{N}_0}$ recursively on the basis of the backward component of the FBSDE without jumps \eqref{FBSDE_without_jump} alone (that is, with the forward component $\{\widetilde{X}_s^{t,{\bf x}}:\,s\in [t,T]\}$ in common for all $m\in \mathbb{N}_0$), as follows:
\begin{equation}
\widetilde{Y}^{0,t,{\bf x}}_s=\Phi(\widetilde{X}^{t,{\bf x}}_T)
-\int_s^T\widetilde{Z}^{0,t,{\bf x}}_rdB_r,\label{Y0}
\end{equation}
and then, for $m\in\mathbb{N}$,
\begin{equation}\label{Ym}
\widetilde{Y}^{m,t,{\bf x}}_s=\Phi(\widetilde{X}^{t,{\bf x}}_T)+\int_s^T f(r,\widetilde{X}_r^{t,{\bf x}},\widetilde{Y}^{m-1,t,{\bf x}}_r,\widetilde{Z}^{m-1,t,{\bf x}}_r,\widetilde{\Gamma}^{m-1,t,{\bf x}}_r)dr-\int_s^T\widetilde{Z}^{m,t,{\bf x}}_rdB_r 
+ \int_s^T \int_{\mathbb{R}_0^{d_{\bf z}}}\widetilde{U}^{m-1,t,{\bf x}}_r({\bf z})\lambda(r,\widetilde{X}^{t,{\bf x}}_r)\nu(d{\bf z})dr.
\end{equation}
Note that there exist unique solutions $\{(\widetilde{Y}^m,\widetilde{Z}^m,\widetilde{U}^m)\}_{m\in\mathbb{N}_0}$ to \eqref{Y0} and \eqref{Ym}, due to Assumption \ref{standing assumption}.
Next, thanks to the imposed smoothness and boundedness conditions on $w_m$ and $\nabla w_m$, the nonlinear Feynman-Kac formula yields
\begin{gather*}
\widetilde{Y}^{m,t,{\bf x}}_s=w_{m}(s,\widetilde{X}_s^{t,{\bf x}}),\ \ \widetilde{Z}^{m,t,{\bf x}}_s=
(\nabla w_{m} \sigma)(s,\widetilde{X}_s^{t,{\bf x}}),\nonumber\\
\widetilde{U}^{m,t,{\bf x}}_s({\bf z})=w_m(s,\widetilde{X}_s^{t,{\bf x}}+h(s,\widetilde{X}_s^{t,{\bf x}},{\bf z}))-w_m(s,\widetilde{X}_s^{t,{\bf x}}),
\\
\widetilde{\Gamma}^{m,t,{\bf x}}_s=\int_{\mathbb{R}_0^{d_{\bf z}}}\widetilde{U}^{m,t,{\bf x}}_s({\bf z})\lambda(s,\widetilde{X}_s^{t,{\bf x}})\nu(d{\bf z})=\gamma(s,\widetilde{X}_s^{t,{\bf x}};\widetilde{U}^{m,t,{\bf x}}),
\end{gather*}
for $m\in\mathbb{N}_0$ and $s\in [t,T]$.
By taking expectation of \eqref{Y0} and \eqref{Ym} with the aid of the representations above, we get the desired results \eqref{w_0_E} and \eqref{w_m_E} as probabilistic representations of the PDEs \eqref{PIDE_0} and \eqref{PIDE_m}.
\end{proof}

\begin{proof}[Proof of Theorem \ref{thm1}]
As in the numerical experiments of Section \ref{section numerical illustrations}, we focus on the interval $[0,T]$ and the approximation of $Y_0^{0,{\bf x}}=u(0,{\bf x})$ in accordance with \eqref{nFK_jump} (that is, the solution at time $0$).
To ease the notation, we write $\widetilde{f}(t,{\bf x},{\bf y},{\bf z},\gamma):=f(t,{\bf x},{\bf y},{\bf z},\gamma)+\gamma$ and reserve $C$ for a positive constant which changes its value from line to line. 

First, we restrict ourselves to the case where the terminal time is bounded by a suitable $\delta \in (0,1]$, that is, $T<\delta$.
To estimate $|(u-w_m)(t,{\bf x})|$, observe that
\begin{align*}
&(u-w_m)(t,{\bf x})\\
&=\mathbb{E}\bigg[\int_t^T \bigg(\widetilde{f}(s,\widetilde{X}^{t,{\bf x}}_s,u(s,\widetilde{X}^{t,{\bf x}}_s),(\nabla u \sigma)(s,\widetilde{X}^{t,{\bf x}}_s),
\gamma(s,\widetilde{X}_s^{t,{\bf x}};u)
)\\
&-\widetilde{f}(s,\widetilde{X}^{t,{\bf x}}_s,w_{m-1}(s,\widetilde{X}^{t,{\bf x}}_s),(\nabla w_{m-1} \sigma)(s,\widetilde{X}^{t,{\bf x}}_s),\int_{\mathbb{R}^{d_{\bf z}}_0}(w_{m-1}(s,\widetilde{X}^{t,{\bf x}}_s+h(s,\widetilde{X}^{t,{\bf x}}_s,z))-w_{m-1}(s,\widetilde{X}^{t,{\bf x}}_s))\lambda(s,\widetilde{X}^{t,{\bf x}}_s)v(d{\bf z}))\bigg)ds\bigg],
\end{align*}
in accordance with the probabilistic representations \eqref{u_E} and \eqref{w_m_E}.
Then, the Lipschitz continuity of the driver $f$ leads to
\begin{align*}
|(u-w_m)(t,{\bf x})|
&\leq C_{\rm Lip}[f]\mathbb{E}\bigg[\int_t^T \bigg(|(u-w_{m-1})(s,\widetilde{X}^{t,{\bf x}}_s)|+|(\nabla (u-w_{m-1}) \sigma)(s,\widetilde{X}^{t,{\bf x}}_s)|\\
&\qquad +\int_{\mathbb{R}^{d_{\bf z}}_0}(|(u-w_{m-1})(s,\widetilde{X}^{t,{\bf x}}_s+h(s,\widetilde{X}^{t,{\bf x}}_s,{\bf z}))|+|(u-w_{m-1})(s,\widetilde{X}^{t,{\bf x}}_s)|)\lambda(s,\widetilde{X}^{t,{\bf x}}_s)v(d{\bf z})\bigg)ds\bigg]\\
&\leq C_{\rm Lip}[f]T(1+2\eta) \sup_{(t,{\bf x})\in [0,T]\times \mathbb{R}^{d_{\bf x}}}|(u-w_{m-1})(t,{\bf x})|+C_{\rm Lip}[f]T\sup_{(t,{\bf x})\in [0,T]\times \mathbb{R}^{d_{\bf x}}}|(\nabla (u-w_{m-1})\sigma)(t,{\bf x})|,
\end{align*}
where we have used the upper bound $\sup_{(t,{\bf x})\in [0,T]\times \mathbb{R}^{d_{\bf x}}}|\lambda(t,{\bf x})|\leq \eta$ for the second inequality.
Taking supremum of both sides with respect to time and spatial variables, we get 
\begin{align}
&\sup_{(t,{\bf x})\in [0,T]\times\mathbb{R}^{d_{\bf x}}}|(u-w_m)(t,{\bf x})|\nonumber\\
&\qquad \leq C_{\rm Lip}[f]T(1+2\eta) \sup_{(t,{\bf x})\in [0,T]\times \mathbb{R}^{d_{\bf x}}}|(u-w_{m-1})(t,{\bf x})|+C_{\rm Lip}[f]T\sup_{(t,{\bf x})\in [0,T]\times \mathbb{R}^{d_{\bf x}}}|(\nabla (u-w_{m-1})\sigma)(t,{\bf x})|\nonumber\\
&\qquad \leq C_{\rm Lip}[f]T(1+2\eta) \sup_{(t,{\bf x})\in [0,T]\times \mathbb{R}^{d_{\bf x}}}|(u-w_{m-1})(t,{\bf x})|+CC_{\rm Lip}[f]T\sup_{(t,{\bf x})\in [0,T]\times \mathbb{R}^{d_{\bf x}}}|\nabla (u-w_{m-1})(t,{\bf x})|,\label{lh_u_sup}
\end{align}
for some positive $C$.
To estimate the second term, we use the following representation due to \cite[Theorem 4.2]{Mall_rep}: 
\begin{multline*}
\nabla(u-w_m)(t,{\bf x})
=\mathbb{E}\bigg[\int_t^T \bigg(\widetilde{f}(s,\widetilde{X}^{t,{\bf x}}_s,u(s,\widetilde{X}^{t,{\bf x}}_s),(\nabla u \sigma)(s,\widetilde{X}^{t,{\bf x}}_s),
\gamma(s,\widetilde{X}^{t,{\bf x}}_s;u)
)\\
-\widetilde{f}(s,\widetilde{X}^{t,{\bf x}}_s,w_{m-1}(s,\widetilde{X}^{t,{\bf x}}_s),(\nabla w_{m-1} \sigma)(s,\widetilde{X}^{t,{\bf x}}_s),
\gamma(s,\widetilde{X}^{t,{\bf x}}_s;w_{m-1})
)\bigg)N_{s,t}^{t,{\bf x}}ds\bigg],
\end{multline*}
with the matrices $J_{t\rightarrow s}^{t,{\bf x}}:=\nabla_{\bf x}^{\top}\widetilde{X}^{t,{\bf x}}_s$ and $N_{s,r}^{t,{\bf x}}:=(s-r)^{-1}(\int_t^s (\sigma^{-1}(v,\widetilde{X}_v^{t,{\bf x}}))^{\top} J_{t\rightarrow v}^{t,{\bf x}}dB_v)^{\top}(J_{t\rightarrow r}^{t,{\bf x}})^{-1}$ for $0\leq t\leq s\leq T$,
with 
\begin{align}
\mathbb{E}\left[\left|N_{s,r}^{t,{\bf x}}\right|^{2}\right]\leq \frac{C}{s-r}.\label{N_l2_est}
\end{align}
for some positive $C$.
By the Cauchy-Schwartz inequality, it holds that
\begin{align*}
|\nabla(u-w_m)(t,{\bf x})|
&\leq \int_t^T \mathbb{E}\bigg[\bigg|\widetilde{f}(s,\widetilde{X}^{t,{\bf x}}_s,u(s,\widetilde{X}^{t,{\bf x}}_s),(\nabla u \sigma)(s,\widetilde{X}^{t,{\bf x}}_s),
\gamma(s,\widetilde{X}^{t,{\bf x}}_s;u)
)\\
&\qquad \qquad -\widetilde{f}(s,\widetilde{X}^{t,{\bf x}}_s,w_{m-1}(s,\widetilde{X}^{t,{\bf x}}_s),(\nabla w_{m-1}\sigma)(s,\widetilde{X}^{t,{\bf x}}_s),
\gamma(s,\widetilde{X}^{t,{\bf x}}_s;w_{m-1})
)\bigg|^2\bigg]^{1/2}\mathbb{E}\left[\left|N_{s,t}^{t,{\bf x}}\right|^2\right]^{1/2}ds\\
&\leq CC_{\rm Lip}[f]\int_t^T \mathbb{E}\bigg[\bigg(|(u-w_{m-1})(s,\widetilde{X}^{t,{\bf x}}_s)|+|(\nabla (u-w_{m-1})\sigma)(s,\widetilde{X}^{t,{\bf x}}_s)|\nonumber\\
&\quad +\int_{\mathbb{R}^{d_{\bf z}}_0}\left(|(u-w_{m-1})(s,\widetilde{X}^{t,{\bf x}}_s+h(s,\widetilde{X}^{t,{\bf x}}_s,z))|+|(u-w_{m-1})(s,\widetilde{X}^{t,{\bf x}}_s)|\right)\lambda(s,\widetilde{X}^{t,{\bf x}}_s)v(d{\bf z})\bigg)^2\bigg]^{1/2}(s-t)^{-1/2}ds\nonumber\\
&\leq 2CC_{\rm Lip}[f]\sqrt{T-t}\left((1+2\eta) \sup_{(t,{\bf x})\in [0,T]\times \mathbb{R}^{d_{\bf x}}}|(u-w_{m-1})(t,{\bf x})|+\sup_{(t,{\bf x})\in [0,T]\times \mathbb{R}^{d_{\bf x}}}|(\nabla (u-w_{m-1})\sigma)(t,{\bf x})|\right).
\end{align*}
where the second inequality holds by the Lipschitz continuity of the driver $f$ and the estimate \eqref{N_l2_est}.
Thus, by taking supremum of both sides with respect to time and spatial variables, it holds that
\begin{multline}
\sup_{(t,{\bf x})\in [0,T]\times\mathbb{R}^{d_{\bf x}}}\left|\nabla(u-w_m)(t,{\bf x})]\right|\\
\leq  CC_{\rm Lip}[f]\sqrt{T}(1+2\eta) \sup_{(t,{\bf x})\in [0,T]\times \mathbb{R}^{d_{\bf x}}}|(u-w_{m-1})(t,{\bf x})|+CC_{\rm Lip}[f]\sqrt{T}\sup_{(t,{\bf x})\in [0,T]\times \mathbb{R}^{d_{\bf x}}}|\nabla (u-w_{m-1})(t,{\bf x})|.\label{lh_du_sup}
\end{multline}
Therefore, combining \eqref{lh_u_sup} and \eqref{lh_du_sup} yields
\begin{multline*}
\sup_{{\bf x}\in\mathbb{R}^{d_{\bf x}}}|(u-w_m)(0,{\bf x})|+\sqrt{T}\sup_{{\bf x}\in\mathbb{R}^{d_{\bf x}}}|\nabla(u-w_m)(0,{\bf x})|\\
\leq  CC_{\rm Lip}[f](1+2\eta)T\left[\sup_{(t,{\bf x})\in [0,T]\times \mathbb{R}^{d_{\bf x}}}|(u-w_{m-1})(t,{\bf x})|+\sup_{(t,{\bf x})\in [0,T]\times \mathbb{R}^{d_{\bf x}}}|\nabla (u-w_{m-1})(t,{\bf x})|\right].
\end{multline*}
Due to $T<\sqrt{T}$ for $T\leq \delta\leq 1$, we have
\begin{multline*}
\sup_{(t,{\bf x})\in [0,T]\times \mathbb{R}^{d_{\bf x}}}|(u-w_m)(t,{\bf x})|+\sqrt{T}\sup_{(t,{\bf x})\in [0,T]\times \mathbb{R}^{d_{\bf x}}}|\nabla(u-w_m)(t,{\bf x})|\\
\leq  CC_{\rm Lip}[f](1+2\eta)\sqrt{T}\left[\sup_{(t,{\bf x})\in [0,T]\times \mathbb{R}^{d_{\bf x}}}|(u-w_{m-1})(t,{\bf x})|+\sqrt{T}\sup_{(t,{\bf x})\in [0,T]\times \mathbb{R}^{d_{\bf x}}}|\nabla (u-w_{m-1})(t,{\bf x})|\right].
\end{multline*}

From this point on, we handle both cases of $\eta>0$ (with jumps) and $\eta=0$ (without jumps) altogether (as discussed in Remark \ref{remark bender denk 2}).
To ease the notation, we write $\mathbbm{1}(\eta=0)$ for the indicator function, which returns $1$ if $\eta=0$ and $0$ otherwise. 
For every $T$ satisfying $T\in (0,(\mathbbm{1}(\eta=0)+1\land\eta^2)/(4(1+2\eta)^2C^2C^2_{\rm Lip}[f]))$, we get
\begin{align*}
&\sup_{(t,{\bf x})\in [0,T]\times \mathbb{R}^{d_{\bf x}}}|(u-w_m)(t,{\bf x})|+\sqrt{T}\sup_{(t,{\bf x})\in [0,T]\times \mathbb{R}^{d_{\bf x}}}|\nabla(u-w_m)(t,{\bf x})|\nonumber\\
&\qquad \qquad \leq 2^{-1}(\mathbbm{1}(\eta=0)+1\land\eta)\left[\sup_{(t,{\bf x})\in [0,T]\times \mathbb{R}^{d_{\bf x}}}|(u-w_{m-1})(t,{\bf x})|+\sqrt{T}\sup_{(t,{\bf x})\in [0,T]\times \mathbb{R}^{d_{\bf x}}}|\nabla (u-w_{m-1})(t,{\bf x})|\right]\\
&\qquad \qquad \qquad \qquad\vdots\nonumber\\
&\qquad \qquad \leq 2^{-m}(\mathbbm{1}(\eta=0)+1\land\eta^m)\left[\sup_{(t,{\bf x})\in [0,T]\times \mathbb{R}^{d_{\bf x}}}|(u-w_{0})(t,{\bf x})|+\sqrt{T}\sup_{(t,{\bf x})\in [0,T]\times \mathbb{R}^{d_{\bf x}}}|\nabla (u-w_{m-1})(t,{\bf x})|\right],
\end{align*}
by recursion.
Since the functions $u$, $\nabla u$, $w_0$ and $\nabla w_0$ are bounded due to Assumption \ref{standing assumption}, it holds that
\begin{equation}\label{final result of first step}
\sup_{(t,{\bf x})\in [0,T]\times \mathbb{R}^{d_{\bf x}}}|(u-w_m)(t,{\bf x})|
\leq C2^{-m}(\mathbbm{1}(\eta=0)+1\land\eta^m).
\end{equation}

It remains to extend the result above to a general terminal time $T>0$.
To this end, set a partition $0=t_0<t_1<\cdots<t_n=T$ of the interval $[0,T]$ with a sufficiently large $n$ so that $t_{k+1}-t_k\leq \delta$ for all $k\in \{0,1,\cdots,n-1\}$, along which we rewrite the solution of the PIDE \eqref{PIDE} as $u(t_n,\widetilde{X}_{t_n}^{0,{\bf x}})=\Phi(\widetilde{X}_{t_n}^{0,{\bf x}})$ and then, backwards for $k\in \{n-1,\cdots,0\}$,
\begin{align*}
&u(r,\widetilde{X}_r^{t_k,{\bf x}})
=u(t_{k+1},\widetilde{X}_{t_{k+1}}^{t_k,{\bf x}})\\
&\quad +\int_r^{t_{k+1}}\widetilde{f}(s,\widetilde{X}_s^{t_k,{\bf x}},u(s,\widetilde{X}_s^{t_k,{\bf x}}),(\nabla u \sigma)(s,\widetilde{X}_s^{t_k,{\bf x}}),\int_{\mathbb{R}^{d_{\bf z}}_0}(u(s,\widetilde{X}_s^{t_k,{\bf x}}+h(s,\widetilde{X}_s^{t_k,{\bf x}},{\bf z}))-u(s,\widetilde{X}_s^{t_k,{\bf x}}))\lambda(s,\widetilde{X}^{t_k,{\bf x}}_s)v(d{\bf z}))ds\nonumber\\
&\quad +\int_r^{t_{k+1}}\sigma\nabla u(s,\widetilde{X}_s^{t_k,{\bf x}})dB_s,\quad r\in[t_k,t_{k+1}),
\end{align*}
and, in a similar manner, rewrite the solution of the PDE \eqref{PIDE_m} as $w_m(t_n,\widetilde{X}_{t_n}^{0,{\bf x}})=\Phi(\widetilde{X}_{t_n}^{0,{\bf x}})$ and then, backwards for $k\in \{n-1,\cdots,0\}$,
\begin{align*}
&w_m(r,\widetilde{X}_r^{t_k,{\bf x}})=w_m(t_{k+1},\widetilde{X}_{t_{k+1}}^{t_k,{\bf x}})\\
&\quad +\int_r^{t_{k+1}}\widetilde{f}(s,\widetilde{X}_s^{t_k,{\bf x}},w_{m-1}(s,\widetilde{X}_s^{t_k,{\bf x}}),(\nabla w_{m-1} \sigma)(s,\widetilde{X}_s^{t_k,{\bf x}}),\int_{\mathbb{R}^{d_{\bf z}}_0}(w_{m-1}(s,\widetilde{X}_s^{t_k,{\bf x}}+h(s,\widetilde{X}_s^{t_k,{\bf x}},z))-w_{m-1}(s,\widetilde{X}_s^{t_k,{\bf x}}))\lambda(s,\widetilde{X}^{t_k,{\bf x}}_s)v(d{\bf z}))ds\nonumber\\
&\quad +\int_r^{t_{k+1}}\sigma\nabla w_{m-1}(s,\widetilde{X}_s^{t_k,{\bf x}})dB_s,\ \ \ \  r\in[t_k,t_{k+1}),
\end{align*}
recursively for $m\in \mathbb{N}$.
Since $t_{k+1}-t_k\leq \delta$ for all $k\in \{0,1,\cdots,n-1\}$, it holds by \eqref{final result of first step} that 
\[
  \sup_{(t,{\bf x})\in [t_k,t_{k+1}]\times \mathbb{R}^{d_{\bf x}}}|(u-w_m)(t,{\bf x})|\leq C2^{-m}(\mathbbm{1}(\eta=0)+1\land\eta^m),\quad k\in \{0,1,\cdots,n-1\},
\]
which clearly yields the desired result.
\end{proof}

\small
\bibliographystyle{abbrv}
\bibliography{bsde.bib}

\newpage
\setcounter{page}{1}
\begin{center}
{\bf \Large Supplementary materials: A forward scheme with machine learning for forward-backward SDEs with jumps by decoupling jumps}
\end{center}


For convenience in coding, we provide the pseudocode for the two machine learning methods that we have developed in Section \ref{section forward scheme}.
We first describe the the least squares Monte Carlo method (Section \ref{subsection LSMC}), where $M$ and $J$ denote the numbers of iid paths and basis functions, respectively.

\begin{algorithm}[H]
\footnotesize
\caption{The least squares Monte Carlo method (Section \ref{subsection LSMC})}\label{Algo_LSMC}
\begin{algorithmic}[1]
\Require number of paths $M\in\mathbb{N}$, number of basis functions $J\in\mathbb{N}$, basis functions $v_j^{\mathrm{basis}}:\mathbb{R}^{d_{\bf x}}\rightarrow\mathbb{R}$ for $j\in\{1,\cdots,J\}$, $\{\beta_{p,k,j}^{{\cal U}}\}_{j\in \{1,\cdots,J\}}\subset\mathbb{R}^{d_{\bf y}}$, and $\{\beta_{p,k,j}^{{\cal V}}\}_{j\in \{1,\cdots,J\}}\subset\mathbb{R}^{d_{\bf y}\times d_0}$ for $k\in\{1,\cdots,n-1\}$ and $p\in\{0,1,\cdots,m\}$
\State Simulate $\Delta B_{t_{k+1}}^\ell:=B_{t_{k+1}}^\ell-B_{t_k}^\ell\sim\mathcal{N}(0_{d_0},(T/n)\mathbb{I}_{d_0})$ and $\overline{X}_{t_k}^{(n),\ell}$ for $k\in \{0,1,\cdots,n\}$ and $\ell\in \{1,\cdots,M\}$
\State Compute ${\cal U}_{p,n}^{\mathrm{LS}}(\overline{X}_{t_n}^{(n),\ell}) = \Phi(\overline{X}_{t_{n}}^{(n),\ell})$ and ${\cal V}_{p,n}^{{\mathrm{LS}}}(\overline{X}_{t_n}^{(n),\ell}) = (\nabla\Phi\sigma)(t_n,\overline{X}_{t_{n}}^{(n),\ell})$ for $\ell\in \{1,\cdots,M\}$ and $p\in \{0,1,\cdots,m\}$
\For{$k=n-1$ to $1$}
\State Solve the following two minimization problems:
\State $\{\beta_{0,k,j}^{{\cal U}}\}_{j\in \{1,\cdots,J\}}\in \argmin_{\{\beta_{j}\}_{j\in \{1,\cdots,J\}}\subset\mathbb{R}^{d_{\bf y}}}M^{-1}\sum_{\ell=1}^M|\sum_{j=1}^J\beta_jv_j^{\mathrm{basis}}(\overline{X}_{t_k}^{(n),\ell})- \Phi(\overline{X}_{t_{n}}^{(n),\ell}) |^2$
\State $\{\beta_{0,k,j}^{{\cal V}}\}_{j\in \{1,\cdots,J\}}\in \argmin_{\{\beta_{j}\}_{j\in \{1,\cdots,J\}}\subset\mathbb{R}^{d_{\bf y}\times d_0}}M^{-1}\sum_{\ell=1}^M|\sum_{j=1}^J\beta_jv_j^{\mathrm{basis}}(\overline{X}_{t_k}^{(n),\ell})-\Phi(\overline{X}_{t_{n}}^{(n),\ell}) (\Delta B_{t_{k+1}}^\ell)^{\top}/\Delta_{n}|^2$
\State Compute ${\cal U}_{0,k}^{\mathrm{LS}}(\overline{X}_{t_k}^{(n),\ell})=\sum_{j=1}^J\beta_{0,k,j}^{{\cal U}}v_j^{\mathrm{basis}}(\overline{X}_{t_k}^{(n),\ell})$ and ${\cal V}_{0,k}^{\mathrm{LS}}(\overline{X}_{t_k}^{(n),\ell})=\sum_{j=1}^J\beta_{0,k,j}^{{\cal V}}v_j^{\mathrm{basis}}(\overline{X}_{t_k}^{(n),\ell})$
\EndFor
\State Compute ${\cal U}_{0,0}^{\mathrm{LS}}({\bf x})=M^{-1}\sum_{\ell=1}^M\Phi(\overline{X}_{t_{n}}^{(n),\ell})$
\For{$p=1$ to $m$}
\For{$k=n-1$ to $1$}
\State Solve the following two minimization problems:
\State $\{\beta_{p,k,j}^{{\cal U}}\}_{j\in \{1,\cdots,J\}}\in \argmin_{\{\beta_{j}\}_{j\in \{1,\cdots,J\}}\subset\mathbb{R}^{d_{\bf y}}}M^{-1}\sum_{\ell=1}^M|\sum_{j=1}^J\beta_jv_j^{\mathrm{basis}}(\overline{X}_{t_k}^{(n),\ell})$\\
\hspace{4em}$-(\Phi(\overline{X}_{t_{n}}^{(n),\ell})+\sum_{i=k+1}^{n}\Delta_{n}f({\cal U}_{p-1,i}^{\mathrm{LS}}(\overline{X}_{t_{i}}^{(n),\ell}),{\cal V}_{p-1,i}^{\mathrm{LS}}(\overline{X}_{t_{i}}^{(n),\ell}),\gamma(t_{i},\overline{X}_{t_{i}}^{(n),\ell};{\cal U}_{p-1,i}^{\mathrm{LS}}))+\sum_{i=k+1}^{n-1}\Delta_{n}\gamma(t_{i},\overline{X}_{t_{i}}^{(n),\ell};{\cal U}_{p-1,i}^{\mathrm{LS}})
)|^2$
\State $\{\beta_{p,k,j}^{{\cal V}}\}_{j\in \{1,\cdots,J\}}\in \argmin_{\{\beta_{j}\}_{j\in \{1,\cdots,J\}}\subset\mathbb{R}^{d_{\bf y}\times d_0}}M^{-1}\sum_{\ell=1}^M|\sum_{j=1}^J\beta_jv_j^{\mathrm{basis}}(\overline{X}_{t_k}^{(n),\ell})$\\
\hspace{4em}$-(\Phi(\overline{X}_{t_{n}}^{(n),\ell})+\sum_{i=k+1}^{n}\Delta_{n}f({\cal U}_{p-1,i}^{\mathrm{LS}}(\overline{X}_{t_{i}}^{(n),\ell}),{\cal V}_{p-1,i}^{\mathrm{LS}}(\overline{X}_{t_{i}}^{(n),\ell}),\gamma(t_{i},\overline{X}_{t_{i}}^{(n),\ell};{\cal U}_{p-1,i}^{\mathrm{LS}}))+\sum_{i=k+1}^{n}\Delta_{n}\gamma(t_{i},\overline{X}_{t_{i}}^{(n),\ell};{\cal U}_{p-1,i}^{\mathrm{LS}})
)(\Delta B_{t_{k+1}}^\ell)^{\top}/\Delta_{n}|^2$
\State Compute ${\cal U}_{p,k}^{\mathrm{LS}}(\overline{X}_{t_k}^{(n),\ell})=\sum_{j=1}^J\beta_{p,k,j}^{{\cal U}}v_j^{\mathrm{basis}}(\overline{X}_{t_k}^{(n),\ell})$ and ${\cal V}_{p,k}^{\mathrm{LS}}(\overline{X}_{t_k}^{(n),\ell})=\sum_{j=1}^J\beta_{p,k,j}^{{\cal V}}v_j^{\mathrm{basis}}(\overline{X}_{t_k}^{(n),\ell})$
\EndFor
\State Compute ${\cal U}_{p,0}^{\mathrm{LS}}({\bf x})=M^{-1}\sum_{\ell=1}^M( \Phi(\overline{X}_{t_{n}}^{(n),\ell}) +\sum_{i=1}^{n}\Delta_{n}f({\cal U}_{p-1,i}^{\mathrm{LS}}(\overline{X}_{t_{i}}^{(n),\ell}),{\cal V}_{p-1,i}^{\mathrm{LS}}(\overline{X}_{t_{i}}^{(n),\ell}),\gamma(t_{i},\overline{X}_{t_{i}}^{(n),\ell};{\cal U}_{p-1,i}^{\mathrm{LS}}))+\sum_{i=1}^{n}\Delta_{n}\gamma(t_{i},\overline{X}_{t_{i}}^{(n),\ell};{\cal U}_{p-1,i}^{\mathrm{LS}}))$
\EndFor
\State Return ${\cal U}_{m,0}^{\mathrm{LS}}({\bf x})$ as an approximation of $w_m(0,{\bf x})$
\end{algorithmic}
\end{algorithm}

\newpage
Next, we provide the pseudocode for the neural network-based scheme (Section \ref{subsection neural network}).
Note that here, $M$ and $J$ denote the batch size and number of train-steps, respectively, unlike in the least squares Monte Carlo method (Algorithm \ref{Algo_LSMC}).

\begin{algorithm}[H]
\footnotesize
\caption{The neural network-based scheme (Section \ref{subsection neural network})}\label{Algo_NN}
\begin{algorithmic}[1]
\Require{batch size $M\in\mathbb{N}$, train steps $J\in\mathbb{N}$, learning rates $\iota(j)$, $\nu\in\mathbb{N}$, and $\Theta_{p,k}^{{\cal U},j},\Theta_{p,k}^{{\cal V},j}\in\mathbb{R}^{\nu}$ for $p\in \{0,1,\cdots,m\}$ and $j\in \{0,1,\cdots,J\}$}
\State Generate $\Delta B_{t_{i+1}}^{\ell,j}:=B_{t_{i+1}}^{\ell,j}-B_{t_i}^{\ell,j}\sim\mathcal{N}(0_{d_0},(T/n)\mathbb{I}_{d_0})$ and $\overline{X}_{t_{i+1}}^{(n),\ell,j}$ for $i\in \{0,1,\cdots,n-1\}$, $\ell\in \{1,\cdots,M\}$ and $j\in\{0,1,\ldots,J-1\}$
\State Compute ${\cal U}_{p,n}^{\mathrm{NN}}(\overline{X}_{t_n}^{(n),\ell,j}) = \Phi(\overline{X}_{t_{n}}^{(n),\ell,j})$, ${\cal V}_{p,n}^{{\mathrm{NN}}}(\overline{X}_{t_n}^{(n),\ell,j}) = (\nabla\Phi\sigma)(t_n,\overline{X}_{t_{n}}^{(n),\ell,j})$, $p\in\{0,1,\ldots,m\}$, $\ell\in\{1,\ldots,M\}$ and $j\in\{0,1,\ldots,J-1\}$
\For{$k=n-1$ to $0$}
\For{$j=0$ to $J-1$}
\State Compute the loss functions: \\
\hspace{2em} $L_j^{{\cal U}}(\Theta_{0,k}^{{\cal U},j})=M^{-1}\sum_{\ell=1}^M|{\rm NN}^{{\cal U}}(\overline{X}_{t_k}^{(n),\ell,j};\Theta_{0,k}^{{\cal U},j})- \Phi(\overline{X}_{t_{n}}^{(n),\ell,j}) |^2$\\
\hspace{2em} $L_j^{{\cal V}}(\Theta_{0,k}^{{\cal V},j})=M^{-1}\sum_{\ell=1}^M|{\rm NN}^{{\cal V}}(\overline{X}_{t_k}^{(n),\ell,j};\Theta_{0,k}^{{\cal V},j})- \Phi(\overline{X}_{t_{n}}^{(n),\ell,j}) (\Delta B_{t_{k+1}}^{\ell,j})^{\top}/\Delta_{n}|^2$
\State Update $\Theta_{0,k}^{{\cal U},j}\rightarrow \Theta_{0,k}^{{\cal U},j+1}$ and $\Theta_{0,k}^{{\cal V},j}\rightarrow \Theta_{0,k}^{{\cal V},j+1}$ by Adam optimizer with the learning rate $\iota(j)$ to minimize $L_j^{{\cal U}}(\Theta_{0,k}^{{\cal U},j})+L_j^{{\cal V}}(\Theta_{0,k}^{{\cal V},j})$
\EndFor
\State Set ${\cal U}_{0,k}^{\mathrm{NN}}={\rm NN}^{{\cal U}}(\cdot ;\Theta_{0,k}^{{\cal U},J})$ and ${\cal V}_{0,k}^{\mathrm{NN}}={\rm NN}^{{\cal V}}(\cdot\ ;\Theta_{0,k}^{{\cal V},J})$
\EndFor
\For{$p=1$ to $m$}
\For{$j=0$ to $J-1$}
\For{$k=n-1$ to $0$}
\State Compute the loss functions:\\
\hspace{5em}$L_j^{{\cal U}}(\Theta_{p,k}^{{\cal U},j})=M^{-1}\sum_{\ell=1}^M|{\rm NN}^{{\cal U}}(\overline{X}_{t_k}^{(n),\ell,j};\Theta_{p,k}^{{\cal U},j})$\\
\hspace{6em}$-(\Phi(\overline{X}_{t_{n}}^{(n),\ell,j})+\sum_{i=k+1}^{n}\Delta_{n}f({\cal U}_{p-1,i}^{\mathrm{NN}}(\overline{X}_{t_{i}}^{(n),\ell,j}),{\cal V}_{p-1,i}^{\mathrm{NN}}(\overline{X}_{t_{i}}^{(n),\ell,j}),\gamma(t_{i},\overline{X}_{t_{i}}^{(n),\ell,j};{\cal U}_{p-1,i}^{\mathrm{NN}}))+\sum_{i=k+1}^{n}\Delta_{n}\gamma(t_{i},\overline{X}_{t_{i}}^{(n),\ell,j};{\cal U}_{p-1,i}^{\mathrm{NN}}))|^2$\\
\hspace{5em}$L_j^{{\cal V}}(\Theta_{p,k}^{{\cal V},j})=M^{-1}\sum_{\ell=1}^M|{\rm NN}^{{\cal V}}(\overline{X}_{t_k}^{(n),\ell,j};\Theta_{p,k}^{{\cal V},j})$\\
\hspace{6em}$-(\Phi(\overline{X}_{t_{n}}^{(n),\ell,j})+\sum_{i=k+1}^{n}\Delta_{n}f({\cal U}_{p-1,i}^{\mathrm{NN}}(\overline{X}_{t_{i}}^{(n),\ell,j}),{\cal V}_{p-1,i}^{\mathrm{NN}}(\overline{X}_{t_{i}}^{(n),\ell,j}),\gamma(t_{i},\overline{X}_{t_{i}}^{(n),\ell,j};{\cal U}_{p-1,i}^{\mathrm{NN}}))+\sum_{i=k+1}^{n}\Delta_{n}\gamma(t_{i},\overline{X}_{t_{i}}^{(n),\ell,j};{\cal U}_{p-1,i}^{\mathrm{NN}}))(\Delta B_{t_{k+1}}^{\ell,j})^{\top}/\Delta_{n}|^2$
\State Update $\Theta_{p,k}^{{\cal U},j}\rightarrow \Theta_{p,k}^{{\cal U},j+1}$ and $\Theta_{p,k}^{{\cal V},j}\rightarrow \Theta_{p,k}^{{\cal V},j+1}$ by Adam optimizer with the learning rate $\iota(j)$ to minimize $L_j^{{\cal U}}(\Theta_{p,k}^{{\cal U},j})+L_j^{{\cal V}}(\Theta_{p,k}^{{\cal V},j})$
\EndFor
\State Set ${\cal U}_{p,k}^{\mathrm{NN}}={\rm NN}^{{\cal U}}(\cdot\ ;\Theta_{p,k}^{{\cal U},J})$ and ${\cal V}_{p,k}^{\mathrm{NN}}={\rm NN}^{{\cal V}}(\cdot\ ;\Theta_{p,k}^{{\cal V},J})$
\EndFor
\EndFor
\State Return ${\cal U}_{m,0}^{\mathrm{NN}}({\bf x})$ as an approximation of $w_m(0,{\bf x})$
\end{algorithmic}
\end{algorithm}

\end{document}